\documentclass[12pt]{amsart}
\usepackage{geometry}                
\usepackage{array}
\usepackage{graphicx}
\usepackage{amssymb}
\usepackage{epstopdf}
\usepackage{multirow}
\newtheorem{theorem}{Theorem}[section]
\newtheorem{lemma}[theorem]{Lemma}
\newtheorem{corollary}[theorem]{Corollary}

\newtheorem{definition}{Definition}[section]
\usepackage[hyphens]{url} 
\usepackage{graphicx} 
\usepackage{tikz}
\usepackage{comment}
\usepackage{amsthm}

\usepackage[ruled]{algorithm2e}
\usepackage[caption=false]{subfig}
\ifpdf
  \DeclareGraphicsExtensions{.eps,.pdf,.png,.jpg}
\else
  \DeclareGraphicsExtensions{.eps}
\fi
\newcommand\restr[2]{{
  \left.\kern-\nulldelimiterspace 
  #1 
  \vphantom{\big|} 
  \right|_{#2} 
  }}
  
\begin{document}
\title[Equivalent extensions of HJB equations on hypersurfaces]{Equivalent extensions of Hamilton-Jacobi-Bellman equations on hypersurfaces}

\author{Lindsay Martin}\thanks{Department of Mathematics, The University of Texas at Austin, TX 78712, USA. E-mail: lmartin@math.utexas.edu}
\author{Richard Tsai} \thanks{Department of Mathematics and Oden Institute for Computational Engineering and Sciences, The University of Texas at Austin, TX 78712, USA. E-mail: ytsai@math.utexas.edu}
\subjclass[2010]{Primary 70H20,49J20, 35R01 58E10, 65N06}

\maketitle
\begin{abstract}
 We present a new formulation for the computation of solutions of a class of Hamilton Jacobi Bellman (HJB) equations on closed smooth surfaces of co-dimension one. 
 For the class of equations considered in this paper,
 the viscosity solution of the HJB equation is equivalent to the value function of a corresponding optimal control problem. 
 In this work, we extend the optimal control problem given on the surface to an equivalent one defined in a sufficiently thin narrow band of the co-dimensional one surface. The extension is done appropriately so that the corresponding HJB equation, in the narrow band, has a unique viscosity solution which is identical to the constant normal extension of the value function of the original optimal control problem. With this framework,  one can easily use existing (high order) numerical methods developed on Cartesian grids to solve HJB equations on surfaces, with a computational cost that scales with the dimension of the surfaces.
 This framework also provides a systematic way for solving HJB equations on the unstructured point clouds that are sampled from the surface.
\end{abstract}
\section{Introduction}\label{sec:CPmap}

Hamilton-Jacobi-Bellman equations have many applications in optimal control, seismology, geometrical optics, etc.  From the solutions of HJB equations on surfaces, the corresponding characteristics curves can be extracted and used in many applications. Some examples include mesh generation \cite{HochRascle2002}, path planning \cite{KumarVladimirsky2010,Tsitsiklis95}, and brain mapping \cite{AlbertoSapiro2001, Lengletetal2009}. 
The equations are fully nonlinear and classical solutions typically do not exist. The unique viscosity solution \cite{CrandallLions83} is often sought after.
Sophisticated algorithms have been developed for computing the viscosity solution of HJB equations defined in Euclidean space.

Let $\Omega \subset \mathbb{R}^n$ where $n=2$ or $3$ be a bounded and connected open set with smooth closed boundary $\Gamma=\partial \Omega$. 
Our goal is to compute solutions of the following HJB equation defined on smooth surfaces, with given Dirichlet boundary conditions:
\begin{align}\label{surfHam}
\min_{\mathbf{a}\in A_\mathbf{x}} \bigg\{ r(\mathbf{x},\mathbf{a})+\nabla_\Gamma u(\mathbf{x}) \cdot {f}(\mathbf{x},\mathbf{a})\mathbf{a} \bigg \}&=0, \; \; \; \mathbf{x} \in \Gamma\backslash \mathcal{T} \\
u(\mathbf{x})&=g(\mathbf{x}), \; \; \mathbf{x} \in \mathcal{T}.\label{surfHambc}
\end{align}
The set $A_\mathbf{x}:=S^{n-1} \cap T_\mathbf{x}\Gamma$ is a compact set where $S^{n-1}$ is the $(n-1)$-dimensional unit sphere in $\mathbb{R}^n$ and $T_\mathbf{x}\Gamma$ is the tangent space of $\Gamma$ at the point $\mathbf{x} \in \Gamma$. We define ${A:=S^{n-1}\cap TM(\Gamma)}$ to be the unit tangent bundle of $\Gamma$ where $TM(\Gamma)$ is the tangent bundle of $\Gamma$. Here $r,f:\Gamma \times A \to \mathbb{R}$, $\mathcal{T}$ is a closed subset of $\Gamma$, and $\nabla_\Gamma u$ is the surface gradient on $\Gamma$ \cite{Deckelnick2019}. 

For $\epsilon>0$, define the narrow band of $\Gamma $ by
$$T_\epsilon : = \{\mathbf{z} \in \mathbb{R}^n : \min_{\mathbf{x} \in \Gamma} ||\mathbf{x}-\mathbf{z}||)< \epsilon\}.$$
In this paper, we shall derive a Hamiltonian, $\overline{H}$, and extensions $\overline{\mathcal{T}}$ and $\overline{g}$ of $\mathcal{T}$ and $g$, respectively, such that the viscosity solution to 
\begin{align}\label{extendHam}
\overline{H}(\mathbf{z},\nabla v (\mathbf{z}))&=0, \; \; \; \mathbf{z} \in T_\epsilon \backslash \overline{\mathcal{T}} \\
v(\mathbf{z})&=\overline{g}(\mathbf{z}), \; \; \mathbf{z} \in \overline{\mathcal{T}}
\end{align}
is the constant normal extension of the solution to \eqref{surfHam}-\eqref{surfHambc}, for any positive and sufficiently small $\epsilon$.

Our contribution includes a theory for how optional control problems on surfaces can be extended into ``equivalent" ones defined in a narrow band of the surface. Depending on whether the problem is isotropic or anisotropic, we must take careful consideration of extending the control space. We show for the anisotropic or most general case, the same control space used in the surface problem must be used in defining the extended problem. However, in the isotropic case the extension is equivalent even when we appropriately extend the control space. After defining the extension of optimal control problems, we then extend the corresponding HJB equation defined on the narrow band of the surface. The main advantage of this approach is that to compute solutions of HJB equations on surfaces, we are able to use Cartesian grids and existing methods (including high order methods) which solve HJB equations in Euclidean space. This allows us to avoid unnecessary patching and triangulation for solving the surface HJB equation. We show that in fact the narrow band can be very thin, i.e., it's thickness is of order $h$, where $h$ is the grid size.

Before proceeding, we will first define the closest point mapping that is used in the extensions. Define the closest point mapping, $P_\Gamma: T_\epsilon \to \Gamma$, by
\begin{equation*}
P_{\Gamma}(\mathbf{z}) := \arg \min_{\mathbf{y} \in \Gamma} ||\mathbf{z}-\mathbf{y}||.
\end{equation*}  Here, $\epsilon$ must be chosen so that the closest point mapping is unique.
Then we define the signed distance function to $\Gamma$, $d_\Gamma:T_\epsilon \to \mathbb{R}$,  by 
\begin{equation*}
d_{\Gamma}(\mathbf{z})= \begin{cases}  ||\mathbf{z}-P_{\Gamma}\mathbf{z}|| \; \; &\mbox{if } \mathbf{z} \in \overline{\Omega}^c.
\\ -||\mathbf{z}-P_{\Gamma}\mathbf{z}|| \; \; &\mbox{if } \mathbf{z} \in \Omega
\end{cases}
\end{equation*}
For $|\eta|<\epsilon$, we will define the parallel surface, $\Gamma_\eta$, by 
$$\Gamma_\eta:=\{\mathbf{z} \in T_\epsilon \; | \; d_\Gamma(\mathbf{z})=\eta\}.$$
Closest point mappings are easily derived in the context of level set methods \cite{OsherSethian88}, and there are a variety of fast and high order methods available to compute them from the distance function to $\Gamma$ \cite{Tsitsiklis95, Sethian96, Chengetal02, Tsaietal03, Zhangetal06}, i.e., 
$$P_\Gamma(\mathbf{z})=\mathbf{z}-d_\Gamma(\mathbf{z}) \nabla d_\Gamma(\mathbf{z}).$$

Finally, we define the constant normal extension of function on $\Gamma$, $u:\Gamma \to \mathbb{R}$, to be $\overline{u}:T_\epsilon \to \mathbb{R}$ given by $$\overline{u}(\mathbf{z}):=u(P_\Gamma \mathbf{z}).$$ Again, our goal is to define an extended HJB equation on $T_\epsilon$  
so that the solution is the constant normal extension to the solution of the HJB equation on the surface.

The rest of the paper is organized as follows: In Section \ref{sec:HJBman}, we present the optimal control problem which models controlled motion on a surface and define the related class of HJB equations. We then extend the  control problem and HJB equations on the narrow band, $T_\epsilon$. In Section \ref{sec:Eikonalcase}, we present a special case of the problem when the HJB equation reduces to the Eikonal equation and the distance function on the surface is desired. Finally, we give some numerical simulations in Section \ref{sec:sim}.

\subsection{Previous work: solving HJB equations in Euclidean space} \label{sec:prevworkEuc}

There is extensive work on developing numerical methods to compute solutions to \eqref{surfHam} where the domain is a bounded open set in $\mathbb{R}^n$. 
In the view of a optimal control problem, one can obtain a semi-Lagrangian discretization for HJB equations which gives a large system of coupled nonlinear equations. Extensive studies of semi-Lagrangian techniques are found in \cite{Falcone87,BardiFalcone90, Falcone94,FalconeFerretti94}. Solving the discretized system using fixed-point iterations is expensive, thus fast marching methods (FMMs) and fast sweeping methods (FSMs) were developed.  

FMMs are a variant of Djikstra's method and use a heapsort algorithm to determine the order in which the grid nodes are updated. The algorithm was first developed on Cartesian grids \cite{Tsitsiklis95,Sethian96} to solve Eikonal equations which are a special case of \eqref{surfHam} when the functions $r$ and $f$ are isotropic, i.e., $r(\mathbf{x},\mathbf{a})=r(\mathbf{x})$ and $f(\mathbf{x},\mathbf{a})=f(\mathbf{x})$. The solution to the Eikonal equation is well known to be the distance function to the set $\mathcal{T}$. FMMs have complexity $\mathcal{O}(N \log N)$, where $N$ is the total number of unknowns \cite{RouyTourin92, Tsitsiklis95}.  Based on the  Godunov upwind numerical scheme developed to solve Eikonal equations on triangulated domains in \cite{BarthSethian98},  FMMs were extended to acute triangulated meshes in \cite{KimmelSethian98} . However, FMMs do not directly apply to the anisotropic case. FMMs were applied to compute solutions of a class of ``axis-aligned'' anisotropic HJ equations in \cite{AltonMitchell09} on orthogonal grids, and ordered upwind methods (OUMs) were developed to solve a class of general HJB equations on any triangulated surface \cite{SethianVladimirsky03}. 
In \cite{Carlinietal06,Carlinietal08}, related PDEs are solved using similar fast methods.

FSMs were developed in \cite{Tsaietal03, Zhao05} to compute solutions of a class of strictly convex HJ equations, including Eikonal equation, on Cartesian grids. 
The FSM is an iterative method that relies on an upwind discretization of the PDE and Gauss-Seidel style updates with different orderings of grid nodes.
The idea is to avoid using heapsort and, instead, update the grid nodes in different orderings (sweeps). In each sweep, characteristics that go in similar directions are propagated automatically. In \cite{LiShuetal08}, a FSM was developed for Eikonal equations using a discontinuous Galerkin
(DG) local solver for computing the distance function. 
For more general Hamilton-Jacobi equations, one may use  a Lax-Friedrichs type numerical Hamiltonian \cite{Kaoetal04} on Cartesian grids. The resulting Lax-Friedrichs sweeping scheme is typically more diffusive and requires more iterations. 
FSMs have complexity $\mathcal{O}(N)$ with the caveat that the constant hidden in the notation may be large depending on the characteristics of the PDE.

\subsection{Previous work: solving HJB equations on surfaces}\label{sec:prevworkman} 

Several versions of FMMs and FSMs can be applied to computing the distance function on surfaces. 
As already mentioned, the FMM in \cite{KimmelSethian98} only works for solving Eikonal equations on acute triangulated surfaces. In \cite{SpiraKimmel04}, FMMs equation were extended to solve Eikonal equations on parametric surfaces  where the discretization is done on a Cartesian grid in the parametric plane.  The FSM in \cite{Tsaietal03} and the FMM in \cite{Mirebeau2014}
can compute the distance function on surfaces that are defined as the graph of a smooth function.  Another approach to computing the distance function on surfaces is to solve a corresponding evolutionary HJ equation where the distance function is the steady state solution \cite{Chengetal02}. However, with the addition of the time variable this method is not computationally optimal.

OUMs are applicable to more general anisotropic HJB equations on triangulated surfaces.  
However, if the surface is given in implicit form, triangulation may be unnecessary. We want to avoid triangulation and have the ability to handle general surface representations. Our motivation is to derive equivalent extended HJB equations on a narrow band of the surface such that a variety of meshes and methods can be used to solve the equivalent equations in the narrow band. In particular, the methods developed for Cartesian grids (surveyed in Section \ref{sec:prevworkEuc}) can then be used. 

In \cite{MemoliSapiro01}, the idea of extending the surface by a small offset, $\epsilon$, to a narrow band is used to compute distance functions on surface. On the narrow band, an extended distance function is computed, but it is not the constant normal extension of the distance function on the surface. Therefore, the formulation introduces an analytical error. It is proven that the error between the two distance functions is controlled by $\epsilon$, the width of the narrow band. In order for the method to converge as the grid size goes to zero, it is required that $\epsilon=Ch^\gamma$ where $h$ is the grid size and $\gamma \in (0,1).$ The advantage of our framework is that there no analytical error introduced. Our formulation is convergent with respect to the grid size, $h$, and for narrow bands of order $h$. We are also able to perform high order computations for the distance function on surfaces. 

The idea of extending functions or differential operators 
defined on surfaces, via closest point mapping, to the embedding Euclidean space, is used in \cite{RuuthMerriman08, MacdonaldRuuth09, ChuTsai18} to compute solutions to parabolic and elliptic PDEs. Our approach is inspired by the work in \cite{ChuTsai18} where variational integrals defined originally on surfaces are extended to ones defined in the narrow band, and the Euler-Lagrange equations of the extended problem are solved by standard numerical methods. Due to the way the extensions are defined, the resulting solution to each equation is automatically the constant normal extension of the solution to the variational problem on the surface.  


\subsection{Computing PDEs ``on point clouds"}
Point clouds arise in many applications including facial recognition, manufacturing, medicine, and geosciences and can be acquired easily through modern sensing devices such as laser scanners and cell phones. 
A typical approach, particularly predominant in the computer graphics community,  is to create triangular meshes from the data points and seek to define partial derivatives in certain suitable ways. Among myriad of papers using triangular surface meshes, we mention \cite{BarthSethian98}, an early paper which is the most related to the topic of this paper. 
Another school of approaches aim at solving differential equations on surfaces, using only a finite set of points sampled from the surfaces, without globally reconstructing the surface. 

In \cite{MemoliSapiro05}, the algorithm developed in \cite{MemoliSapiro01} is extended to compute distance functions on surfaces represented as point clouds. In \cite{LiangZhao2013}, a framework for computing solutions to PDEs on point clouds based on a local approximation of the manifold was presented.  A local mesh algorithm was developed in \cite{Laieta2013} that solves PDEs on point clouds where any of the existing methods valid on triangular meshes discussed in Section \ref{sec:prevworkman} can be used to compute the solutions to HJB equations once the local mesh is determined. 

Our new proposed non-parametric formulation also 
provides a convenient way to solve equations ``on point clouds." That is we can solve HJB equations which are defined on a smooth closed surface, but the only available information about the surface is a finite set of points that are reasonably distributed over the surface.  
The discussion of the implementation of our method applied to point clouds is in Section \ref{sec:PointCloud}. There, it is assumed that there is no noise perturbing the point cloud in the surface normal directions, and that the point clouds are evenly distributed over the surfaces.
The generalization of the proposed algorithm to more general point clouds
is the subject of another paper.

\section{Static Hamilton-Jacobi-Bellman equations on surfaces}\label{sec:HJBman}

\subsection{Modeling controlled motion on a surface}\label{Gamma}
First, we define the optimal control problem from which the HJB equations can be derived. For each $\mathbf{x} \in \Gamma$, let $A_\mathbf{x}$ be set of all unit tangent vectors of $\Gamma$ at $\mathbf{x}$, i.e., $$A_\mathbf{x}=S^{n-1} \cap T_\mathbf{x}\Gamma.$$ Note each $A_\mathbf{x}$ is a compact set in $\mathbb{R}^n$ for each $\mathbf{x} \in \Gamma$. Define $A:=\bigcup_{\mathbf{x}\in \Gamma} A_\mathbf{x}$, then $A=S^{n-1}\cap TM(\Gamma)$ where $TM(\Gamma)$ is the tangent bundle of $\Gamma$.  The set of admissible controls is given by: 
$$\mathcal{A}=\{ \mbox{measurable functions } \mathbf{a}:[0,\infty) \to A\}.$$ We are interested in the trajectories governed by the dynamical system:
\begin{align}\label{dynsys0}
\frac{d\mathbf{y}}{dt}(t)&={f}(\mathbf{y}(t),\mathbf{a}(t))\mathbf{a}(t), \; t>0 \\  \label{dynbc0} \mathbf{y}(0)&=\mathbf{x},
\end{align}
where ${f}:\Gamma \times A \to \mathbb{R}$ represents the velocity and $\mathbf{a}(t) \in A_{\mathbf{y}(t)}$, which ensures that $\mathbf{y}(t) \in \Gamma$ for all $t\geq0$. Denote the solution to \eqref{dynsys0}-\eqref{dynbc0} (which exists under the assumption \eqref{eq:A1} below) by $\mathbf{y}_\mathbf{x}(t,\mathbf{a}(t))$. For each $\mathbf{x} \in \Gamma$ define the following subset of the admissible controls:

$$\mathcal{A}_\mathbf{x}:=\{\mathbf{a}(\cdot) \in \mathcal{A}  \; | \; \mathbf{y}_\mathbf{x}(t,\mathbf{a}(t)) \in \Gamma \mbox{ for all } t\geq0\}.$$ Then the requirement that $\mathbf{a}(t) \in A_{\mathbf{y}(t)}$ for all $t \geq 0$ in the dynamical system \eqref{dynsys0}-\eqref{dynbc0} is equivalent to $\mathbf{a}(\cdot)\in \mathcal{A}_\mathbf{x}.$ Refer to $\mathbf{y}(\cdot)$ that satisfy (\ref{dynsys0})-(\ref{dynbc0}) where $\mathbf{a}(\cdot) \in \mathcal{A}_\mathbf{x}$ as an admissible path on $\Gamma$. For the rest of the paper, we will suppress writing the dependence of $\mathbf{y}_\mathbf{x}(t, \mathbf{a}(t))$ on $\mathbf{a}(\cdot)$ when there is no confusion in the context. 
 
 Let $r : \Gamma \times A  \to \mathbb{R}$ denote a running cost per unit time and $g:\mathcal{T} \to \mathbb{R}$ be an exit time penalty when the state reaches a closed target set $\mathcal{T} \subset \Gamma$. The total cost associated with initial state $\mathbf{x}$ and control $\mathbf{a}(\cdot)\in \mathcal{A}_\mathbf{x}$ to reach $\mathcal{T}$ is given by 
\begin{equation}
C(\mathbf{x},\mathbf{a}(\cdot))=\int_0^T r(\mathbf{y}_\mathbf{x}(t),\mathbf{a}(t))dt +g(\mathbf{y}_\mathbf{x}(T)),
\end{equation}
where $T=\min \{ t \; | \; \mathbf{y}_\mathbf{x}(t) \in \mathcal{T}\}.$

We have the following assumptions on $f,r,$ and $g$:
\begin{equation}\label{eq:A1}\begin{cases} f \mbox{ and } r \mbox{ are Lipschitz continuous in their first arguments}. \\
g \mbox{ is lower semicontinuous and } \min_\mathcal{T} g < \infty.\\
\mbox{There exists constants } R_1,R_2,F_1, \mbox{ and } F_2 \mbox{ such that: } \\ 0<R_1\leq r(\mathbf{x},\mathbf{a}) \leq R_2,  \; \mbox{for all } \mathbf{x} \in \Gamma \mbox{ and } \mathbf{a} \in A, \\ 0<F_1\leq f(\mathbf{x},\mathbf{a}) \leq F_2, \; \mbox{for all } \mathbf{x} \in \Gamma \mbox{ and } \mathbf{a} \in A. \tag{A1}
\end{cases}\end{equation} The value function $u: \Gamma \to \mathbb{R}$ is defined to be the minimal total cost to $\mathcal{T}$ starting at $\mathbf{x}$ and is given by 
\begin{equation}\label{valfunman}
u(\mathbf{x})=\inf_{\mathbf{a}(\cdot) \in \mathcal{A}_\mathbf{x}} C(\mathbf{x}, \mathbf{a}(\cdot)).
\end{equation}

Under the assumption \eqref{eq:A1}, an optimal control does not necessarily exist. If for the given $f$ and $r$, we have that the set $V(\mathbf{x})=\{f(\mathbf{x},\mathbf{a})\mathbf{a}/r(\mathbf{x}, \mathbf{a}) \; | \; \mathbf{a} \in A_\mathbf{x}\}$ is strictly convex for each $\mathbf{x} \in \Gamma$, then an optimal control is guaranteed \cite{BardiDolcetta97}. This property is trivially satisfied in the case of isotropic $f$ and $r$. Regardless of whether an optimal control exists, we can still derive the corresponding HJB equation. 

The dynamic programming principle \cite{BardiDolcetta97} for the value function states that for sufficiently small $\tau >0$,
\begin{equation} \label{bellman0}
u(\mathbf{x})=\inf_{\mathbf{a}(\cdot) \in \mathcal{A}_\mathbf{x}} \bigg \{ \int_0^\tau r(\mathbf{y}_\mathbf{x}(t),\mathbf{a}(t))dt+u(\mathbf{y}_\mathbf{x}(\tau)) \bigg \}.
\end{equation}

The corresponding Hamilton-Jacobi-Bellman equation on $\Gamma$ is
\begin{equation} \label{HJBman0}
\min_{\mathbf{a}\in A_\mathbf{x}} \bigg\{ r(\mathbf{x},\mathbf{a})+\nabla_\Gamma u(\mathbf{x}) \cdot {f}(\mathbf{x},\mathbf{a})\mathbf{a} \bigg \}=0, \; \mathbf{x} \in \Gamma\backslash \mathcal{T}.
\end{equation}
The boundary condition is 
\begin{equation}\label{HJBmanbc0} u(\mathbf{x})=g(\mathbf{x}), \; \; \mathbf{x}\in \mathcal{T}.
\end{equation}
Note that in \eqref{HJBman0} we can take the minimum since $A_\mathbf{x}$ is compact.

In general, classical solutions of \eqref{HJBman0}-\eqref{HJBmanbc0} do not exist, and the unique viscosity solution is sought after \cite{CrandallLions83}. A detailed discussion of viscosity solutions of Hamilton-Jacobi equations on manifolds can be found in \cite{Mantegazza2003}. Under the assumption (A1), it is a classic result that the value function $u$ coincides with the unique viscosity solution of \eqref{HJBman0}-\eqref{HJBmanbc0} \cite{BardiDolcetta97}.
In the case of isotropic running cost and speed, then \eqref{HJBman0} reduces to the Eikonal equation on the surface:
\begin{equation}\label{Eikman0}
||\nabla_\Gamma u(\mathbf{x})||=\frac{r(\mathbf{x})}{f(\mathbf{x})}.
\end{equation}
We will expand more on the Eikonal equation in the next section.

\subsection{Extension to $T_\epsilon$}\label{sec:extTeps}
We derive a Hamiltonian, $\overline{H}$, on $T_\epsilon$ such that if $v: T_\epsilon \to \mathbb{R}$ is the unique viscosity solution of $$\overline{H}(\mathbf{z},\nabla v(\mathbf{z}))=0$$ 
with boundary conditions defined appropriately, then $$v(\mathbf{z})=\overline{u}(\mathbf{z}):=u(P_\Gamma \mathbf{z}),$$ 
where $u$ is the viscosity solution to \eqref{HJBman0}-\eqref{HJBmanbc0}.

First, we extend the optimal control problem on the surface to one in the narrow band, $T_\epsilon$. We start by defining a tensor that will be essential in formulating an equivalent optimal control problem on $T_\epsilon$.

\begin{definition}\label{defn:Btensor}
Let $\mathbf{t}_1(\mathbf{z})$ and $\mathbf{t}_2(\mathbf{z})$ be two orthonormal tangent vectors corresponding to the directions that yield the principle curvatures of $\Gamma_\eta$ at a point $\mathbf{z}$. Let $\mathbf{n(\mathbf{z})}$ be the unit normal vector to the tangent plane at $\mathbf{z}$. Let $\sigma_1(\mathbf{z})$ and $\sigma_2(\mathbf{z})$ be the singular values of $P_\Gamma'(\mathbf{z})$, the Jacobian matrix of $P_\Gamma$ at $\mathbf{z}$. 
Then for any real number $\mu$, define the tensor $B$ by
\begin{equation}B(\mathbf{z},\mu)=B_0(\mathbf{z})+\mu B_1(\mathbf{z}), \end{equation}
where 
\begin{align*}
B_0(\mathbf{z})&=\sigma_1^{-1}(\mathbf{z})\mathbf{t}_1(\mathbf{z}) \otimes \mathbf{t}_1(\mathbf{z}) +\sigma_2^{-1}(\mathbf{z}) \mathbf{t}_2(\mathbf{z}) \otimes \mathbf{t}_2(\mathbf{z}),\\
B_1(\mathbf{z})&=\mathbf{n}(\mathbf{z})\otimes \mathbf{n}(\mathbf{z}).
\end{align*} 
\end{definition} 
In the above definition, $B_0$ corresponds to a weighted orthogonal projection onto the tangent plane of $\Gamma_\eta$ at $\mathbf{z}$, and $B_1$ is the projection along the normal of $\Gamma_\eta$ passing through $\mathbf{z}$. 
In \cite{KublikTsai16}, it is proven that $$\sigma_i=1-d_\Gamma(\mathbf{z})\kappa_i(\mathbf{z})$$ where $\kappa_1(\mathbf{z})$ and $\kappa_2(\mathbf{z})$ are the principal curvatures of the parallel surface $\Gamma_\eta$. 
Note that when $\mathbf{z} \in \Gamma$, $\sigma_1=\sigma_2=1.$ For the sake of notation, we will suppress the dependence of the tangent vectors, normal vectors, and singular values on the point $\mathbf{z} \in T_\epsilon$. 

Assuming $\epsilon$ small enough, we have that the tangent planes coincide for $\mathbf{z} \in T_\epsilon$ and $\mathbf{x} \in \Gamma$ if $P_\Gamma \mathbf{z}=\mathbf{x}$. Thus, for the extended optimal control problem, the set of compact control values at each $\mathbf{z} \in T_\epsilon$ is $A_{P_\Gamma \mathbf{z}}.$ Then $A$ and $\mathcal{A}$ are as in Section \ref{Gamma}. Now, consider the following extended dynamical system in $T_\epsilon$: 
\begin{align}\label{dyn0}
\frac{d{\mathbf{y}}}{dt}(t)&=\overline{f}({\mathbf{y}}(t),\mathbf{a}(t))B({\mathbf{y}}(t),\mu)\mathbf{a}(t), \; t>0 \\ \label{dyn20}
{\mathbf{y}}(0)&=\mathbf{z},\; \mathbf{z} \in T_\epsilon.
\end{align}  Here $\overline{f}: T_\epsilon \times A \to \mathbb{R}$ is given by $\overline{f}(\mathbf{z},\mathbf{a})=f(P_\Gamma \mathbf{z}, \mathbf{a}),$ and we have that $\mathbf{a}(t) \in A_{\mathbf{y}(t)}$ for all $t\geq0.$ This last requirement ensures that the trajectory remains in $\Gamma_\eta$ for all $t \geq 0$ where $\eta=d_\Gamma(\mathbf{z})$. The inclusion of the tensor $B(\mathbf{y}(t),\mu)$ in the dynamical system above adjusts the velocity according to the curvature of the surface. This has the implication that equivalent paths will have the same total cost.

Denote the solution to \eqref{dyn0}-\eqref{dyn20} (which exists under the assumption \eqref{eq:B1} below) by $\mathbf{y}_{\mathbf{z}}(t,\mathbf{a}(t))$. For each $\mathbf{z} \in T_\epsilon$ with $d_\Gamma(\mathbf{z})=\eta$, we have the following subset of admissible controls:
$$\mathcal{A}_{\mathbf{z}}:=\{\mathbf{a} \in \mathcal{A} \; |\; \mathbf{y}_{\mathbf{z}}(t,\mathbf{a}(t)) \in \Gamma_\eta\}.$$
Then the requirement that $\mathbf{a}(t) \in A_{\mathbf{y}_\mathbf{z}(t)}$ for all $t \geq 0$ in the dynamical system \eqref{dyn0}-\eqref{dyn20} is equivalent to $\mathbf{a}(\cdot)\in \mathcal{A}_\mathbf{z}.$ Since $A_\mathbf{x}=A_{P_\Gamma \mathbf{z}}$, we have $\mathcal{A}_{\mathbf{z}}=\mathcal{A}_{P_\Gamma \mathbf{z}}$.   Refer to $\mathbf{y}(\cdot)$ that satisfy (\ref{dyn0})-(\ref{dyn20}) where $\mathbf{a}(\cdot) \in \mathcal{A}_\mathbf{z}$ as an admissible path on $T_\epsilon$. Again, for the rest of the paper we will write suppress writing the dependence of $y_{\mathbf{z}}(t,\mathbf{a}(t))$ on $\mathbf{a}(\cdot)$. There is no confusion between the notation for admissible paths on $\Gamma$ and $T_\epsilon$. The initial point will indicate which dynamical system the path solves.

Note that if the initial point $\mathbf{z} \in \Gamma$, $\eta=0$ and $$B(\mathbf{y}_\mathbf{z}(t), \mu)\mathbf{a}(t)=\mathbf{a}(t),$$  since $\mathbf{a}(t) \in T_{\mathbf{y}_\mathbf{z}(t)}$ and $\sigma_1=\sigma_2=1.$ Thus, \eqref{dyn0}-\eqref{dyn20} reduces to the dynamical system on $\Gamma$ given in \eqref{dynsys0}-\eqref{dynbc0}, and an admissible path on $T_\epsilon$ with $\mathbf{a}(\cdot) \in \mathcal{A}_\mathbf{z}$ is also an admissible path on $\Gamma$ with the same $\mathbf{a}(\cdot) \in \mathcal{A}_\mathbf{z}$.  

Now, let  $\overline{\mathcal{T}}=\{\mathbf{z} \in T_\epsilon \; | \; P_\Gamma \mathbf{z} \in \mathcal{T}\}$. Let the running cost and exit time penalty on $T_\epsilon$ be given by $\overline{r}: T_\epsilon \times A \to \mathbb{R}$ where $\overline{r}(\mathbf{z},\mathbf{a})=r(P_\Gamma \mathbf{z}, \mathbf{a})$ and $\overline{g}:T_\epsilon \to \mathbb{R}$ where $\overline{g}(\mathbf{z})=g(P_\Gamma \mathbf{z})$, respectively.  Then the total cost associated with initial state $\mathbf{z}$ and control ${\mathbf{a}}(\cdot)\in \mathcal{A}_\mathbf{z}$ to reach $\overline{\mathcal{T}}$ is given by 
\begin{align*}
\overline{C}(\mathbf{z},\mathbf{a}(\cdot))&:=\int_0^{\overline{T}} \overline{r}( \mathbf{{y}_\mathbf{z}}(t),\mathbf{a}(t)) dt +\overline{g}( \mathbf{{y}}_\mathbf{z}(\overline{T}))
\end{align*}
where $\overline{T}=\min \{ t \;| \; \mathbf{{y}_\mathbf{z}}(t) \in \overline{\mathcal{T}}\}$. Note that if $\mathbf{z} \in \Gamma$, then $\overline{C}(\mathbf{z},\mathbf{a}(\cdot))=C(\mathbf{z},\mathbf{a}(\cdot)).$ 
The value function $v:T_\epsilon \to \mathbb{R}$ is the minimal total cost to $\overline{\mathcal{T}}$ starting at $\mathbf{z}$ and is given by
\begin{equation}\label{valueeps0}
v(\mathbf{z})=\inf_{{\mathbf{a}}(\cdot) \in {\mathcal{A}_\mathbf{z}}} \overline{C}(\mathbf{z},{\mathbf{a}}(\cdot)).
\end{equation}

We have that the assumption \ref{eq:A1} implies:

\begin{equation}\label{eq:B1}\begin{cases} \mbox{The mappings } (\mathbf{z},\mathbf{a}) \mapsto \overline{r}( \mathbf{z},\mathbf{a}) \mbox{ and }  (\mathbf{z},\mathbf{a}) \mapsto \overline{f}( \mathbf{z},\mathbf{a})B(\mathbf{z},\mu)\mathbf{a} \\ \mbox{are Lipschitz continuous in their first arguments.}   \\
\overline{g} \mbox{ is lower semicontinuous and } \min_{\overline{\mathcal{T}}} \overline{g} < \infty.\\
 0<R_1\leq \overline{r}(\mathbf{z},\mathbf{a}) \leq R_2,  \; \mbox{for all } \mathbf{z} \in T_\epsilon \mbox{ and } \mathbf{a} \in A, \\
  0<G_1\leq \overline{g}(\mathbf{z}) \leq G_2, \;  \mbox{for all } \mathbf{z} \in \overline{\mathcal{T}},\\
\mbox{There exists constants } \overline{F}_1,\overline{F}_2 \mbox{ such that: } \\0<\overline{F}_1\leq ||\overline{f}(\mathbf{z},\mathbf{a})B(\mathbf{z},\mu)\mathbf{a}|| \leq \overline{F}_2, \; \mbox{for all } \mathbf{z} \in T_\epsilon \mbox{ and } \mathbf{a} \in A.  \tag{B1}
\end{cases}\end{equation}
Again, while \eqref{eq:B1} does not guarantee the existence of the optimal control, the assumption does ensure that the value function and corresponding viscosity solution of the HJB equation coincide. If we also assume the set $V(\mathbf{x})$ is strictly convex, then $\overline{V}(\mathbf{z})=\{\overline{f}(\mathbf{z},\mathbf{a})B(\mathbf{z},\mu)\mathbf{a}/\overline{r}(\mathbf{z}, \mathbf{a}) \; | \; \mathbf{a} \in A_\mathbf{z}\}$ is strictly convex for each $z \in T_\epsilon$. Thus, if the optimal control exists in the problem on $\Gamma$, it also exists for extended problem on $T_\epsilon$.

Before deriving the corresponding HJB equation, we prove that the value function, $v$, is the constant normal extension of the value function of the optimal control problem defined on $\Gamma$.  The following theorem states that two admissible paths, $\mathbf{y}_{\mathbf{z}_1}(\cdot)$ and $\mathbf{y}_{\mathbf{z}_2}(\cdot)$, on $T_\epsilon$ with equivalent initial points in $T_\epsilon$ and equivalent controls stay equivalent for all time.  
\begin{theorem}\label{radial0} Suppose that $\mathbf{z}_1,\mathbf{z}_2 \in T_\epsilon$ such that  $P_\Gamma \mathbf{z}_1=P_\Gamma \mathbf{z}_2$. Let  ${\mathbf{a}_1}(\cdot) \in {\mathcal{A}_{\mathbf{z}_1}}$ and ${\mathbf{a}_2}(\cdot) \in {\mathcal{A}_{\mathbf{z}_2}}$ such that ${\mathbf{a}_1}\equiv {\mathbf{a}_2}$. Suppose ${{\mathbf{y}_{\mathbf{z}_1}}:[0,\infty) \to T_\epsilon}$   solves \eqref{dyn0} with  ${\mathbf{y}_{\mathbf{z}_1}}(0)=\mathbf{z}_1$ and ${{\mathbf{y}_{\mathbf{z}_2}}:[0,\infty) \to T_\epsilon}$ solves \eqref{dyn0} with  ${\mathbf{y}_{\mathbf{z}_{2}}}(0)=\mathbf{z}_2$.  Then $P_\Gamma {\mathbf{y}_{\mathbf{z}_1}}(t)=P_\Gamma {\mathbf{y}_{\mathbf{z}_2}}(t)$ for all $t \geq 0$.

\end{theorem}
\begin{proof}
Let $\mathbf{y}_1:[0,\infty) \to \Gamma$ and $\mathbf{y}_2:[0,\infty) \to \Gamma$ be given by $$\mathbf{y}_1(t)=P_\Gamma {\mathbf{y}_{\mathbf{z}_1}}(t)$$ and $$\mathbf{y}_2(t)=P_\Gamma {\mathbf{y}_{\mathbf{z}_2}}(t)$$ for all $t \geq 0$, respectively. We need to show that $\mathbf{y}_1(t)=\mathbf{y}_2(t)$ for all $t\geq 0$. We compute $ \mathbf{y}_1'$. We have the following singular value decomposition for $P_\Gamma '$: $P_\Gamma'=U \Sigma U^T$ where  \[
   U=
  \left[ {\begin{array}{ccc}
   \mathbf{t}_1 &\mathbf{t}_2 & \mathbf{n}
  \end{array} } \right] \mbox{ and }
  \Sigma= \left[ {\begin{array}{ccc}
 
    \sigma_1 & 0 &0\\
    0 & \sigma_2 &0\\
    0 &    0     &0
  \end{array} } \right].
\]

Since $\mathbf{t}_1,$  $\mathbf{t}_2$, and $\mathbf{n}$ are orthonormal, for $i,j=1,2$,

\begin{equation*}
(\mathbf{t}_i \otimes \mathbf{t}_i)(\mathbf{t}_j \otimes \mathbf{t}_j)=
    \begin{cases}
    \mathbf{t}_i \otimes \mathbf{t}_i,~~~&i=j,\\
    0,~~~&i\neq j,
    \end{cases}
\end{equation*}
and for $i=1,2$
\begin{equation*}
(\mathbf{t}_i \otimes \mathbf{t}_i)(\mathbf{n} \otimes \mathbf{n}) =0. 
\end{equation*}
Using the above, we have  
\begin{align*}
\mathbf{y}_1'(t) & = (P_\Gamma \mathbf{{y}_{\mathbf{z}_1}}(t))'\\
&=P_\Gamma ' \frac{d\mathbf{{y}_{\mathbf{z}_1}}}{dt}(t)\\ 
&= P_\Gamma'\overline{f}({\mathbf{y}_{\mathbf{z}_1}}(t),\mathbf{a}_1(t)) B({\mathbf{y}_{\mathbf{z}_1}}(t),\mu)\mathbf{a}_1(t)\\ 
&=P_\Gamma'B({\mathbf{y}_{\mathbf{z}_1}}(t),\mu)\overline{f}({\mathbf{y}_{\mathbf{z}_1}}(t),\mathbf{a}_1(t)) \mathbf{a}_1(t)\\
&= (\mathbf{t}_1 \otimes \mathbf{t}_1+\mathbf{t}_2 \otimes \mathbf{t}_2 ){f}({\mathbf{y}}_1(t),\mathbf{a}_1(t)){\mathbf{a}_1}(t) \\
&= {f}({\mathbf{y}}_1(t),\mathbf{a}_1(t))(\mathbf{t}_1 \otimes \mathbf{t}_1+\mathbf{t}_2 \otimes \mathbf{t}_2 ){{\mathbf{a}_1}}(t)\\
&=f(\mathbf{y}_1(t),\mathbf{a}_1(t))\mathbf{a}_1(t).
\end{align*}  
The last equality is true since $\mathbf{t}_1 \otimes \mathbf{t}_1+\mathbf{t}_2 \otimes \mathbf{t}_2$ is the orthogonal projection of $\mathbf{a}_1(t)$ onto the tangent space at $\mathbf{y}_{\mathbf{z}_1}(t)$ and $\mathbf{a}_1(t) \in A_{\mathbf{y}_{\mathbf{z}_1}(t)} \subset T_{\mathbf{y}_{\mathbf{z}_1}(t)}\Gamma$. Thus $\mathbf{y}_1$ is an admissible path on $\Gamma$ where $\mathbf{y}_1(0)=\mathbf{x}=P_\Gamma \mathbf{z}_1=P_\Gamma \mathbf{z}_2$ and $\mathbf{a}_1(\cdot) \in \mathcal{A}_\mathbf{x}=\mathcal{A}_{P_\Gamma \mathbf{z}_1}.$ By the same reasoning $\mathbf{y}_2$ also an admissible path with $\mathbf{y}_2(0)=\mathbf{x}$ and $\mathbf{a}_2(\cdot) \in \mathcal{A}_\mathbf{x}=\mathcal{A}_{P_\Gamma \mathbf{z}_2}.$  We assume \eqref{eq:A1}. Thus, the solution to \eqref{dynsys0}-\eqref{dynbc0} is unique. Since $\mathbf{a}_1\equiv \mathbf{a}_2$, $\mathbf{y}_1(t)=\mathbf{y}_2(t)$ for all $t \geq 0$. 
\end{proof}

Now, it easily follows that the value function defined on $T_\epsilon$ is the constant normal extension of the value function defined on $\Gamma.$

\begin{corollary}\label{thm:radial0}
If $u:\Gamma \to \mathbb{R}$ is the value function on $\Gamma$ given by \eqref{valfunman}, then 
$$v(\mathbf{z})=\overline{u}(\mathbf{z}):=u(P_\Gamma \mathbf{z})$$
where $v$ is the value function on $T_\epsilon$ given by \eqref{valueeps0}.

\end{corollary}
\begin{proof}

Let $\mathbf{z} \in T_\epsilon$ and $\mathbf{a}(\cdot) \in \mathcal{A}_\mathbf{z}$. Then $\mathbf{a}(\cdot) \in \mathcal{A}_{P_\Gamma \mathbf{z}}$ since $\mathcal{A}_\mathbf{z}=\mathcal{A}_{P_\Gamma \mathbf{z}}$.  Then we have

\begin{align*}
\overline{C}(\mathbf{z},{\mathbf{a}}(\cdot))=\int_0^{\overline{T}} \overline{r}( \mathbf{{y}}_{\mathbf{z}}(t),\mathbf{a}(t)) dt +\overline{g}( {\mathbf{y}_{\mathbf{z}}}(\overline{T}))
\end{align*}
and 
\begin{align*}
\overline{C}(P_\Gamma\mathbf{z},{\mathbf{a}}(\cdot))=\int_0^{{T}} \overline{r}( \mathbf{{y}}_{P_\Gamma\mathbf{z}}(t),\mathbf{a}(t)) dt +\overline{g}( {\mathbf{y}_{P_\Gamma\mathbf{z}}}({T})),
\end{align*}  where ${\mathbf{y}}_{\mathbf{z}}(\cdot)$ and  ${\mathbf{y}}_{P_\Gamma\mathbf{z}}(\cdot)$ are admissible paths on $T_\epsilon$ with ${\mathbf{y}}_{\mathbf{z}}(0)=\mathbf{z}$, ${\mathbf{y}}_{P_\Gamma \mathbf{z}}(0)=P_\Gamma\mathbf{z}$, respectively. Recall, that $\mathbf{y}_{P_\Gamma \mathbf{z}}$ is also an admissible path on $\Gamma$ so that $\overline{C}(P_\Gamma\mathbf{z},{\mathbf{a}}(\cdot))={C}(P_\Gamma\mathbf{z},{\mathbf{a}}(\cdot))$.

Theorem \ref{radial0} implies that $P_\Gamma \mathbf{y}_{\mathbf{z}}(t) = \mathbf{y}_{P_\Gamma \mathbf{z}}(t)$ for all $t \geq 0$.  
Since $$\overline{\mathcal{T}} = \{\mathbf{z} \in T_\epsilon \; | \;P_\Gamma \mathbf{z} \in \mathcal{T} \},$$
we have 
\begin{align*}
\min\{t \; | \; P_\Gamma \mathbf{y}_{\mathbf{z}}(t)= \mathbf{y}_{P_\Gamma \mathbf{z}}(t) \in \mathcal{T}\} & = \min\{ t \; | \; \mathbf{y}_{\mathbf{z}}(t) \in \overline{\mathcal{T}} \}.
\end{align*} 

Therefore, $\overline{T}=T$, and
\begin{align*}
\overline{C}(\mathbf{z},\mathbf{a}(\cdot))
&=\overline{C}(P_\Gamma \mathbf{z}, \mathbf{a}(\cdot))\\
&=C(P_\Gamma \mathbf{z},\mathbf{a}(\cdot)).
\end{align*} 
Finally, we have
\begin{align*}
v(\mathbf{z})&=\inf_{{\mathbf{a}}(\cdot) \in {\mathcal{A}_\mathbf{z}}} \overline{C}(\mathbf{z},{\mathbf{a}}(\cdot))\\
&=\inf_{{\mathbf{a}}(\cdot) \in {\mathcal{A}_{P_\Gamma \mathbf{z}}}} {C}(P_\Gamma \mathbf{z},{\mathbf{a}}(\cdot))\\
&=u(P_\Gamma \mathbf{z}).
\end{align*}
\end{proof}

Now that we have showed that the value function on the narrow band is equivalent to the value function on the surface, we define the Hamilton-Jacobi-Bellman equation on $T_\epsilon$ associated the value function $v$. Again, the dynamic programming principle states that for sufficiently small $\tau>0$, we have  
\begin{equation}\label{bellman20}
v(\mathbf{z})=\inf_{{\mathbf{a}}(\cdot) \in {\mathcal{A}_\mathbf{z}} } \bigg \{ \int_0^\tau \overline{r}( \mathbf{{y}}_{\mathbf{z}}(t),\mathbf{a}(t))dt +v(\mathbf{{y}_{\mathbf{z}}}(\tau))\bigg \}.
\end{equation}
The corresponding Hamilton-Jacobi-Bellman equation on $T_\epsilon$ is
\begin{equation}\label{HJBeps0}
\min_{{\mathbf{a}}\in {A}_\mathbf{z}} \bigg\{ \overline{r}( \mathbf{z},\mathbf{a})+\nabla v(\mathbf{z}) \cdot \overline{f}( \mathbf{z},\mathbf{a})B(\mathbf{z},\mu){\mathbf{a}} \bigg \}=0, \;\; \mathbf{z} \in T_\epsilon \backslash\overline{\mathcal{T}} .
\end{equation}
The boundary condition is 
\begin{equation}\label{HJBepsbc0} v(\mathbf{z})=\overline{g}( \mathbf{z}), \; \; \mathbf{z}\in \overline{\mathcal{T}}.
\end{equation} 

The assumption \eqref{eq:B1} implies that the value function coincides with the unique viscosity solution \cite{BardiDolcetta97}. Corollary \ref{thm:radial0} implies that the viscosity solution of \eqref{HJBeps0}-\eqref{HJBepsbc0} is given by \begin{equation}\label{viscradial} v(\mathbf{z})=u(P_\Gamma \mathbf{z}),\end{equation} where $u$ is the viscosity solution of \eqref{HJBman0}-\eqref{HJBmanbc0} on $T_\epsilon$. An interesting result is that we can prove \eqref{viscradial} without Corollary \ref{thm:radial0}, i.e., using only the derived HJB equations and properties of viscosity solutions. 

We begin with the following lemma that relates the surface gradients for parallel surfaces at equivalent points:
\begin{lemma}\label{gradequiv}
Assume $\epsilon$ is small enough so that $P_\Gamma$ is differentiable, and let $|\eta| <\epsilon$. Given $\phi_0 \in C^1(\Gamma)$, define $\phi_\eta \in C^1(\Gamma_\eta)$  by $\phi_\eta(\mathbf{z}):=\phi_0(P_\Gamma \mathbf{z})$. Then for $\mathbf{z} \in \Gamma_\eta$, 
\begin{equation}\label{surfgrads} \nabla_\Gamma \phi_0(P_\Gamma \mathbf{z})=B(\mathbf{z},\mu) \nabla_{\Gamma_\eta} \phi_\eta (\mathbf{z}).
\end{equation}
\end{lemma}
\begin{proof}
Let $\overline{\phi}_0: T_\epsilon \to \mathbb{R}$ be the constant normal extension of $\phi_0.$ Then $\overline{\phi}_0 \in C^1(T_\epsilon)$ and $\restr{\overline{\phi}_0}{\Gamma_\eta}=\phi_{\eta}$. From Theorem A.1 in \cite{KublikTsai16},  we have  that ${\nabla_\Gamma \phi_0}(P_\Gamma \mathbf{z})=B(\mathbf{z},\mu)\nabla \overline {\phi}_0 (\mathbf{z})$.  Since $\overline{\phi}_0$ is the constant normal extension of $\phi_0$, $\nabla \overline{\phi}_0(\mathbf{z}) \in T_\mathbf{z}\Gamma_\eta$. Thus, if $\mathbf{z} \in \Gamma_\eta$, $\nabla \overline{\phi}_0(\mathbf{z})=\nabla_{\Gamma_\eta} \phi_\eta(\mathbf{z})$, and \eqref{surfgrads} holds.
\end{proof}
The following theorem proves \eqref{viscradial} directly using the definition of viscosity solutions.
\begin{theorem}\label{generalmain} If $u:\Gamma \to \mathbb{R}$ is the viscosity solution to \eqref{HJBman0}-\eqref{HJBmanbc0}, then the constant normal extension of $u$, $\overline{u}: T_\epsilon \to \mathbb{R},$ is the viscosity solution to \eqref{HJBeps0}-\eqref{HJBepsbc0}.
\end{theorem}
\begin{proof} 
We have  $u(\mathbf{x})=g(\mathbf{x})$ for $\mathbf{x}\in \mathcal{T}$ and $P_\Gamma \mathbf{z} \in \mathcal{T}$ if $\mathbf{z} \in  \overline{\mathcal{T}}$. Then for $\mathbf{z} \in \overline{\mathcal{T}}$, 
\begin{align*}
\overline{u}(\mathbf{z})&=u(P_\Gamma \mathbf{z})=g(P_\Gamma \mathbf{z})=\overline{g}(\mathbf{z}).
\end{align*}
Thus, $\overline{u}$ satisfies the boundary conditions \eqref{HJBepsbc0}. Next, we will show that $\overline{u}$ is a viscosity subsolution of \eqref{HJBeps0}.

Assume $\mathbf{z}_0 \in T_\epsilon$ and $\phi \in C^1(T_\epsilon)$ such that $\overline{u} -\phi$ has a strict local maximum at $\mathbf{z}_0$ and  $(\overline{u}-\phi)(\mathbf{z}_0)=0.$ We need to show that 
\begin{equation}\label{subsolution}
\min_{{\mathbf{a}}\in {A}_\mathbf{z}} \bigg\{ \overline{r}( \mathbf{z}_0,\mathbf{a})+\nabla \phi(\mathbf{z}_0) \cdot \overline{f}( \mathbf{z}_0,\mathbf{a})B(\mathbf{z}_0,\mu){\mathbf{a}} \bigg \}\leq0.
\end{equation}

Let $\eta=d_\Gamma(\mathbf{z}_0)$ and $\phi_0: \Gamma \to \mathbb{R}$ be the restriction to $\Gamma$ of the normal extension of $\restr{\phi}{\Gamma_\eta}$. Then we have that $\phi_0 \in C^1(\Gamma)$ and $\phi_\eta:= \restr{\phi}{\Gamma_\eta} \in C^1(\Gamma_\eta)$ such that ${\phi_\eta(\mathbf{z})=\phi_0(P_\Gamma \mathbf{z})}$ for $\mathbf{z}\in \Gamma_\eta$.

First, we will prove 
\begin{equation}\label{square}
B(\mathbf{z}_0,\mu)\nabla \phi(\mathbf{z}_0) \cdot  \mathbf{a} =B(\mathbf{z}_0,\mu)\nabla_{\Gamma_\eta} \phi_\eta(\mathbf{z}_0) \cdot  \mathbf{a}.
\end{equation}
 Since $\mathbf{a} \in A_{\mathbf{z}_0} \subset T_{\mathbf{z}_0}\Gamma$, we have $\mathbf{a} \cdot \mathbf{n}=0$ and $$B(\mathbf{z}_0,\mu)\mathbf{n}\otimes\mathbf{n} \nabla \phi(\mathbf{z}_0)\cdot \mathbf{a}=0.$$ Now, $\nabla_{\Gamma_\eta} \phi_\eta(\mathbf{z}_0)=(I-\mathbf{n}\otimes \mathbf{n})\nabla \phi(\mathbf{z}_0)$. Therefore, \eqref{square} holds.

Lemma \ref{gradequiv}  implies
\begin{equation}\label{triangle}
B(\mathbf{z}_0,\mu)\nabla_{\Gamma_\eta} \phi_\eta(\mathbf{z}_0)=\nabla_\Gamma \phi_0(P_\Gamma \mathbf{z}_0).
\end{equation}
Let $\mathbf{x}_0=P_\Gamma \mathbf{z}_0.$ Then using the fact that $B(\mathbf{z}_0,\mu)$ is symmetric, \eqref{square}, and \eqref{triangle}, we have 
\begin{align*}
\min_{{\mathbf{a}}\in {A}_{\mathbf{z}_0}} \bigg\{ \overline{r}( \mathbf{z}_0,\mathbf{a})+\nabla \phi(\mathbf{z}_0) \cdot &\overline{f}( \mathbf{z}_0,\mathbf{a})B(\mathbf{z}_0,\mu){\mathbf{a}} \bigg \}
\\&=\min_{{\mathbf{a}}\in {A}_{\mathbf{z}_0}} \bigg\{ \overline{r}( \mathbf{z}_0,\mathbf{a})+B(\mathbf{z}_0,\mu)\nabla \phi(\mathbf{z}_0) \cdot \overline{f}( \mathbf{z}_0,\mathbf{a}){\mathbf{a}} \bigg \}\\
&=\min_{{\mathbf{a}}\in {A}_{\mathbf{z}_0}} \bigg\{ \overline{r}( \mathbf{z}_0,\mathbf{a})+B(\mathbf{z}_0,\mu)\nabla_{\Gamma_\eta} \phi_\eta(\mathbf{z}_0) \cdot \overline{f}( \mathbf{z}_0,\mathbf{a}){\mathbf{a}} \bigg \}\\
&=\min_{{\mathbf{a}}\in {A}_{\mathbf{z}_0}} \bigg\{ \overline{r}( \mathbf{z}_0,\mathbf{a})+\nabla_{\Gamma} \phi_0(P_\Gamma \mathbf{z}_0) \cdot \overline{f}( \mathbf{z}_0,\mathbf{a}){\mathbf{a}} \bigg \}\\
&=\min_{{\mathbf{a}}\in {A}_{\mathbf{x}_0}} \bigg\{ {r}( \mathbf{x}_0,\mathbf{a})+\nabla_{\Gamma} \phi_0(\mathbf{x}_0) \cdot {f}( \mathbf{x}_0,\mathbf{a}){\mathbf{a}} \bigg \}.
\end{align*}

Now, $\mathbf{x}_0 \in \Gamma, \phi_0 \in C^1(\Gamma)$ such that $u-\phi_0$ has a strict local maximum at $\mathbf{x}_0$ and $$(u-\phi_0)(\mathbf{x}_0)=(\overline{u}-\phi)(\mathbf{z}_0)=0.$$ Since $u$ is a viscosity subsolution of \eqref{HJBman0}, $$\min_{{\mathbf{a}}\in {A}_{\mathbf{x}_0}} \bigg\{ {r}( \mathbf{x}_0,\mathbf{a})+\nabla_{\Gamma} \phi_0(\mathbf{x}_0) \cdot {f}( \mathbf{x}_0,\mathbf{a}){\mathbf{a}} \bigg \}\leq0.$$ Thus, \eqref{subsolution} holds. The same argument can be applied to show that $\overline{u}$ is a viscosity supersolution of \eqref{HJBman0}.
\end{proof}

In the next section, we will present a special case of the above formulations for when the speed and running cost functions are isotropic. The HJB equation on $\Gamma$ reduces to the Eikonal equation. A noteworthy revelation is that if we ensure that the extended speed and cost functions are isotropic as well, the control values in the extended optimal control problem do not have to be restricted to be in the tangent bundle  of $\Gamma$, i.e., the extended set of controls is $\widetilde{A}=S^{n-1}$. In this setup, the value function is the constant normal extension of the value function on $\Gamma.$ Before we proceeding, we present a counter example to show that $\overline{u}$ is not necessarily the value function of the extended control problem if the extended speed function $\overline{f}$ is anisotropic and $\widetilde{A}=S^{n-1}$.

\subsubsection*{Example}
Consider the following control problems on $\Gamma$ and $T_\epsilon$: Let $${\Gamma=(-1,1) \subset \mathbb{R}}$$ and $$\mathcal{T}=\{1\}.$$ Let $\epsilon>\sqrt{3}/3$, then $$T_\epsilon=(-1,1)\times(-\epsilon,\epsilon)$$ and $$\overline{\mathcal{T}}=\{1\} \times(-\epsilon,\epsilon).$$ For the control problem on $\Gamma$, the set of controls is $A=\{(-1,0),(1,0)\}$ and $\mathcal{A}$ is the usual set of admissible controls. Define $f:\Gamma \times A \to \mathbb{R}$ by $f(\mathbf{x},\mathbf{a})=\sqrt{3}$ and suppose that we have unit running cost, i.e., $r(\mathbf{x},\mathbf{a})=1$. Then it is clear that $u(0)=\sqrt{3}/3$ which is the minimum travel time from $\mathbf{x}=0$ to $\mathcal{T}$ while traveling at speed $\sqrt{3}$.

Now for the extended control problem, let the control values be the set $\widetilde{A}=S^1$. Then the set of admissible controls is given by $$\widetilde{\mathcal{A}}:=\{\mbox{ measurable functions } \widetilde{a}:[0,\infty) \to \widetilde{A}\}.$$ Define $\overline{f}: T_\epsilon \times \widetilde{\mathcal{A}} \to \mathbb{R}$ such that $\overline{f}(\mathbf{z},\widetilde{\mathbf{a}})=\overline{f}(\widetilde{\mathbf{a}})$ and the set $\{\overline{f}(\widetilde{\mathbf{a}}) \widetilde{\mathbf{a}}\; | \;\widetilde{\mathbf{a}} \in \widetilde{A}\}$ is the ellipse given by $x^2+\frac{y^2}{9}=1$ rotated at the origin clockwise by the angle $\pi/3$. Then $\overline{f}(\mathbf{a})=\sqrt{3}$ if $\mathbf{a} \in A$ so that $\overline{f}(\mathbf{z},\mathbf{a})=\overline{f}(P_\Gamma \mathbf{z}, \mathbf{a})$ for $\mathbf{a} \in A$.  Again, assume unit running cost.

Let $v$ be the value function for this extended control problem. If $v=\overline{u}$, then we should have $v((0,0))=u(0)$. However, note that $\overline{f}(\mathbf{z},(\sqrt{3}/2,1/2))=3$ and the Euclidean distance between $(0,0)$ and $(1,\sqrt{3}/3)$ is $2\sqrt{3}/3$. Thus, $$\overline{C}((0,0), \widetilde{a}(\cdot))=\frac{2\sqrt{3}}{9}$$ when $\widetilde{a}(\cdot) \equiv (\sqrt{3}/2,1/2)$. By the definition of the value function we must have $$v((0,0))\leq \frac{2\sqrt{3}}{9} < \frac{\sqrt{3}}{3}= u(0).$$ Therefore, $\overline{u}$ cannot be the value function for the extended control problem.

\section{The Isotropic Case}\label{sec:Eikonalcase}
We derive the optimal control problems and HJB equations corresponding to the case when the speed and cost functions are isotropic. As mentioned at the end of Section \ref{sec:extTeps}, we will show that value function on $T_\epsilon$ is the normal extension of the value function on $\Gamma$ even when we do not restrict the controls to the tangent space in the extended optimal control problem.
\subsection{The controlled dynamics on $\Gamma$ for isotropic speed and cost functions}
The optimal control problem formulation is exactly the same as in Section \ref{Gamma}, and \eqref{HJBman0} reduces to the Eikonal equation on the surface:
\begin{equation}\label{Eikman}
||\nabla_\Gamma u(\mathbf{x})||=\frac{r(\mathbf{x})}{f(\mathbf{x})}.
\end{equation}
We have the HJB equation
\begin{equation}\label{EikHJBman}
H(\mathbf{x}, \nabla_\Gamma u (\mathbf{x}))=r(\mathbf{x})-f(\mathbf{x})||\nabla_\Gamma u(\mathbf{x})||=0, \;\; \mathbf{x}\in \Gamma, 
\end{equation}
and the boundary condition is 
\begin{equation}\label{HJBmanbc} u(\mathbf{x})=g(\mathbf{x}), \; \; \mathbf{x}\in \mathcal{T}.
\end{equation}
Assuming \eqref{eq:A1}, the value function $u$ coincides with the unique viscosity solution of \eqref{EikHJBman}-\eqref{HJBmanbc}. We also have that $V(\mathbf{x})$ is strictly convex for all $\mathbf{x} \in \Gamma$. Thus, an optimal control is guaranteed. 

\subsection{Extension to $T_\epsilon$ in the isotropic case}
We derive the extended Hamiltonian, $\overline{H}$, on $T_\epsilon$ from the extended isotropic optimal control problem such that if $v: T_\epsilon \to \mathbb{R}$ is the unique viscosity solution of $$\overline{H}(\mathbf{x},\nabla v(\mathbf{z}))=0$$ with boundary conditions defined appropriately, then $$v(\mathbf{z})=u(P_\Gamma \mathbf{z}),$$ 
where $u$ is the solution to \eqref{EikHJBman}-\eqref{HJBmanbc}.

The notations $\mathbf{t}_1$, $\mathbf{t}_2$, $\mathbf{n}$, $\sigma_1$, and $\sigma_2$ are as in Section \ref{sec:extTeps}. The extended set of compact control values is $\widetilde{A}=S^{n-1}$, and the set of admissible controls is given by: $$\widetilde{\mathcal{A}}=\{\mbox{measurable functions: }\widetilde{\mathbf{a}}: [0,\infty) \to \widetilde{A}\}.$$ This is where the formulation differs from the previous section. We do not restrict the controls to belong to the tangent bundle of $\Gamma$. We extend the speed, cost, and exit time penalty functions  as above: $\overline{f}(\mathbf{z})=f(P_\Gamma \mathbf{z})$, $\overline{r}(\mathbf{z})=r(P_\Gamma \mathbf{z})$, and $\overline{g}(\mathbf{z})=g(P_\Gamma \mathbf{z})$. The dynamical system is given by 
\begin{align}\label{eq:isodyn}
\frac{d{\mathbf{y}}}{dt}(t)&=\overline{f}({\mathbf{y}}(t))B({\mathbf{y}}(t),\mu)\widetilde{\mathbf{a}}(t), \; t>0 \\ \label{eq:isodynbc}
{\mathbf{y}}(0)&=\mathbf{z},\; \mathbf{z} \in T_\epsilon,
\end{align} 
where $\widetilde{\mathbf{a}}(\cdot) \in \widetilde{\mathcal{A}}$. Admissible paths on $T_\epsilon$ with extended controls are solutions to \eqref{eq:isodyn}-\eqref{eq:isodynbc} denoted by $y_\mathbf{z}(t)$.  Note that the admissible paths are not restricted to the parallel surface, $\Gamma_\eta$, to which the initial point $\mathbf{z}$ belongs.

The total cost function associated to the initial state $\mathbf{z} \in T_\epsilon$ and control $\widetilde{\mathbf{a}}(\cdot) \in \widetilde{\mathcal{A}}$ is the same as in Section \ref{sec:extTeps}.
The value function $v:T_\epsilon \to \mathbb{R}$ is the minimal total cost to $\overline{\mathcal{T}}$ starting at $\mathbf{z}$ and is given by \begin{equation}\label{valueeps} v(\mathbf{z})=\inf_{\widetilde{\mathbf{a}}(\cdot) \in \widetilde{\mathcal{A}}} \overline{C}(\mathbf{z},\widetilde{\mathbf{a}}(\cdot)).\end{equation}

We have that \eqref{eq:A1} implies the following is true:

\begin{equation}\label{eq:C1}\begin{cases} \mbox{The mapping } (\mathbf{z},\widetilde{\mathbf{a}}) \mapsto \overline{f}( \mathbf{z})B(\mathbf{z},\mu)\widetilde{\mathbf{a}} \mbox{ is Lipschitz continuous}\\  \mbox{in the first argument and, } \overline{r} \mbox{ is Lipschitz continuous.}  \\
\overline{g} \mbox{ is lower semicontinuous and } \min_\mathcal{T} \overline{g} < \infty.\\
 0<R_1\leq \overline{r}(\mathbf{z}) \leq R_2,  \; \mbox{for all } \mathbf{z} \in T_\epsilon. \\
\mbox{There exists constants } \overline{F}_1,\overline{F}_2 \mbox{ such that: } \\0<\overline{F}_1\leq ||\overline{f}(\mathbf{z})B(\mathbf{z},\mu)\widetilde{\mathbf{a}}|| \leq \overline{F}_2, \; \mbox{for all } \mathbf{z} \in T_\epsilon \mbox{ and } \widetilde{\mathbf{a}} \in \widetilde{A}.  \tag{C1}
\end{cases}\end{equation}
We have that \eqref{eq:C1} implies the value function, $v$, coincides with the unique viscosity solution of \eqref{EikHJBman}-\eqref{HJBmanbc} \cite{BardiDolcetta97}. We also have that $\overline{V}(\mathbf{z})=\{\overline{f}(\mathbf{z})\widetilde{\mathbf{a}}/\overline{r}(\mathbf{z})\;|\; \widetilde{\mathbf{a}}\in \widetilde{A} \}$ is strictly convex for all $\mathbf{z} \in \Gamma$. Thus, an optimal control is guaranteed. 

We now prove that the value function, $v$, is a constant extension of a function on $\Gamma$. However, unlike in the previous section, it is not immediate from the following proofs that $v=\overline{u}$, where $u$ is the value function on $\Gamma$.
\begin{theorem}\label{radial}
Given $\widetilde{\mathbf{a}}(\cdot) \in \widetilde{\mathcal{A}}$ and $\mathbf{z}_1,\mathbf{z}_2 \in T_\epsilon$ such that $P_\Gamma \mathbf{z}_1=P_\Gamma \mathbf{z}_2$,  let $${{\mathbf{y_{\mathbf{z}_1}}}:[0,\infty) \to T_\epsilon}$$ and $${{\mathbf{y}_{\mathbf{z}_2}}:[0,\infty) \to T_\epsilon}$$   solve \eqref{eq:isodyn}  with  ${\mathbf{y}_{\mathbf{z}_1}}(0)=\mathbf{z}_1$ and  ${\mathbf{y}_{\mathbf{z}_2}}(0)=\mathbf{z}_2$, respectively.  Then $P_\Gamma \mathbf{y}_{\mathbf{z}_1}(t)=P_\Gamma {\mathbf{y}}_{\mathbf{z}_2}(t)$ for all $t \geq 0$.

\end{theorem}

The proof of Theorem \ref{radial} is analogous to the proof of Theorem \ref{radial0} except the control $\widetilde{\mathbf{a}}(\cdot)$ belongs to the extended control space, $\widetilde{\mathcal{A}}$. A consequence of the control belonging to the extended space is that now the projected paths $P_\Gamma \mathbf{y}_{\mathbf{z}}(\cdot)$ may not solve \eqref{dynsys0}. This is due to the fact that $(\mathbf{t}_1 \otimes \mathbf{t}_1+\mathbf{t}_2 \otimes \mathbf{t}_2 )\widetilde{\mathbf{a}}(t)\not =\widetilde{\mathbf{a}}(t)$ if $\widetilde{\mathbf{a}}(t) \not \in T_{P_\Gamma \mathbf{y}_{\mathbf{z}}(t)}$ where $\mathbf{t}_1$ and $\mathbf{t}_2$ are the basis tangent vectors of $T_{P_\Gamma \mathbf{y}_{\mathbf{z}}(t)}.$  We now prove that the extended value function in the isotropic case is constant along the normals of $\Gamma$.

\begin{corollary}\label{thm:radial}
If $P_\Gamma \mathbf{z}_1=P_\Gamma\mathbf{z}_2$ for $\mathbf{z}_1,\mathbf{z}_2 \in T_\epsilon$, then $v(\mathbf{z}_1)=v(\mathbf{z}_2)$.
\end{corollary}
\begin{proof}
Let $\widetilde{\mathbf{a}}(\cdot)\in \widetilde{\mathcal{A}}$. We have  
\begin{align*}
\overline{C}(\mathbf{z}_1,\widetilde{\mathbf{a}}(\cdot))=\int_0^{\overline{T}_1} \overline{r}( \mathbf{{y}}_{\mathbf{z}_1}(t)) dt +\overline{g}( {\mathbf{y}_{\mathbf{z}_1}}(\overline{T}_1))
\end{align*}
and 
\begin{align*}
\overline{C}(\mathbf{z}_2,\widetilde{\mathbf{a}}(\cdot))=\int_0^{\overline{T}_2} \overline{r}( \mathbf{{y}}_{\mathbf{z}_2}(t)) dt +\overline{g}( {\mathbf{y}_{\mathbf{z}_2}}(\overline{T}_2)),
\end{align*}  where ${\mathbf{y}}_{\mathbf{z}_1}(\cdot)$ and  ${\mathbf{y}}_{\mathbf{z}_2}(\cdot)$ are admissible paths on $T_\epsilon$ with extended control, $\widetilde{\mathbf{a}}(\cdot)\in \widetilde{\mathcal{A}}$, and ${{\mathbf{y}}_{\mathbf{z}_1}(0)=\mathbf{z}_1}$, ${\mathbf{y}}_{\mathbf{z}_2}(0)=\mathbf{z}_2$ where ${P_\Gamma \mathbf{z}_1=P_\Gamma\mathbf{z}_2}$.

Theorem \ref{radial} implies that $P_\Gamma \mathbf{y}_{\mathbf{z}_1}(t) = P_\Gamma  \mathbf{y}_{\mathbf{z}_2}(t)$ for all $t \geq 0$.  
Since $$\overline{\mathcal{T}} = \{\mathbf{z} \in T_\epsilon \; | \;P_\Gamma \mathbf{z} \in \mathcal{T} \},$$
we have 
\begin{align*}
\min\{t \; | \; P_\Gamma \mathbf{y}_{\mathbf{z}_1}(t)=P_\Gamma \mathbf{y}_{\mathbf{z}_2}(t) \in \mathcal{T}\} & = \min\{ t \; | \; \mathbf{y}_{\mathbf{z}_1}(t) \in \overline{\mathcal{T}} \} =  \min\{ t \; | \; \mathbf{y}_{\mathbf{z}_2}(t) \in \overline{\mathcal{T}} \}.
\end{align*} 
Therefore, $\overline{T}_1=\overline{T}_2$ and  $\overline C(\mathbf{z}_1, \widetilde{\mathbf{a}}(\cdot))=\overline C(\mathbf{z}_2, \widetilde{\mathbf{a}}(\cdot))$ for $\widetilde{\mathbf{a}}(\cdot) \in \widetilde{\mathcal{A}}$. Now, 
\begin{align*}
v(\mathbf{z}_1)&=\inf_{\widetilde{\mathbf{a}}(\cdot) \in \widetilde{\mathcal{A}}} \overline{C}(\mathbf{z}_1,\widetilde{\mathbf{a}}(\cdot)) =\inf_{\widetilde{\mathbf{a}}(\cdot) \in \widetilde{\mathcal{A}}} \overline{C}(\mathbf{z}_2,\widetilde{\mathbf{a}}(\cdot)) =v(\mathbf{z}_2).
\end{align*}

\end{proof}

In the above proof, it is not immediate that we can take the infimum over $\mathcal{A} \subset \widetilde{\mathcal{A}},$ which would imply that $v$ is the normal extension of $u$, the value function on $\Gamma$. We will use the corresponding HJB equations and Corollary \ref{thm:radial} to show that $v=\overline{u}$.

We define the HJB equation on $T_\epsilon$ associated the value function $v$. Again, the dynamic programming principle states that for sufficiently small $\tau>0$ we have  
\begin{equation}\label{bellman2}
v(\mathbf{z})=\inf_{\widetilde{\mathbf{a}}(\cdot) \in \widetilde{\mathcal{A}} } \bigg \{ \int_0^\tau \overline{r}( \mathbf{y}_\mathbf{z}(t))dt +v(\mathbf{y}_\mathbf{z}(\tau))\bigg \}.
\end{equation}
We get the Hamilton-Jacobi-Bellman equation on $T_\epsilon$:
\begin{equation}\label{HJBeps}
\min_{\widetilde{\mathbf{a}}\in \widetilde{A}} \bigg\{ \overline{r}( \mathbf{z})+\nabla v(\mathbf{z}) \cdot B(\mathbf{z},\mu)\overline{f}( \mathbf{z})\widetilde{\mathbf{a}} \bigg \}=0.
\end{equation}
Since the speed and cost functions are isotropic, \eqref{HJBeps} reduces to an anisotropic Eikonal equation:
\begin{equation}\label{epsEik}
\overline{H}(\mathbf{z},\nabla v(\mathbf{z}))=\overline{r}( \mathbf{z})-\overline{f}( \mathbf{z})||B(\mathbf{z},\mu) \nabla v(\mathbf{z})||=0.
\end{equation}
The boundary condition is 
\begin{equation}\label{HJBepsbc} v(\mathbf{z})=\overline{g}( \mathbf{z}), \; \; \mathbf{z}\in \overline{\mathcal{T}}.\end{equation}

Since the value function coincides with the unique viscosity solution, Corollary \ref{thm:radial} implies that the viscosity solution of \eqref{epsEik}-\eqref{HJBepsbc} is a constant along the normals of $\Gamma$, i.e., there is a function $w:\Gamma \to \mathbb{R}$ such that for $\mathbf{z} \in T_\epsilon$, $$v(\mathbf{z})=\overline{w}(\mathbf{z}):=w(P_\Gamma \mathbf{z}).$$ We now prove that $w=u$ where $u$ is the viscosity solution of \eqref{EikHJBman}-\eqref{HJBmanbc}.

\begin{theorem} 
If $u:\Gamma \to \mathbb{R}$ is the viscosity solution to \eqref{EikHJBman}-\eqref{HJBmanbc}, then the constant normal extension of $u$, $\overline{u}: T_\epsilon \to \mathbb{R},$ is the viscosity solution to \eqref{epsEik}-\eqref{HJBepsbc}.
\end{theorem}
\begin{proof} 
Just as in Theorem \ref{generalmain}, $\overline{u}$ satisfies the boundary conditions \eqref{HJBepsbc}.  Assume for contradiction that $\overline{u}$ is not the viscosity solution to \eqref{epsEik}. But from Corollary \ref{thm:radial} we have that viscosity solution of \eqref{epsEik}-\eqref{HJBepsbc}, $v=\overline{w}$ for some $w \in C^1(\Gamma).$ Therefore, if $\overline{u}$ is not the unique viscosity solution then we have $u \not = w$. For contradiction, we will show that $w$ is the viscosity solution of \eqref{EikHJBman} and thus, $u=w$.

Let $\mathbf{x}_0 \in \Gamma$ and $\phi \in C^1(\Gamma)$ such that $w-\phi$ has a local maximum at $\mathbf{x}_0$. Consider the normal extensions of $w$ and $\phi$, then $\overline{w}-\overline{\phi}$ has a local max at $\mathbf{x}_0$. Since $\overline{w}$ is a viscosity subsolution to \eqref{epsEik} and $\overline{\phi} \in C^1(T_\epsilon)$, we have that 
\begin{equation}
\overline{r}( \mathbf{x}_0)-\overline{f}( \mathbf{x}_0)||B(\mathbf{x}_0,\mu) \nabla \overline{\phi}(\mathbf{x}_0)|| \leq 0.
\end{equation} 
Since $\mathbf{x}_0 \in \Gamma$, $\sigma_1=\sigma_2=1$ and $\nabla \overline{\phi}(\mathbf{x}_0) \in T_{\mathbf{x}_0}\Gamma$. Therefore, $$B(\mathbf{x}_0,\mu)=\mathbf{t}_1\otimes\mathbf{t}_1+\mathbf{t}_2\otimes \mathbf{t}_2+\mu \mathbf{n}\otimes\mathbf{n},$$ and we have that $B(\mathbf{x}_0,\mu)\nabla \overline{\phi}(\mathbf{x}_0)=\nabla \overline{\phi}(\mathbf{x}_0)=\nabla_\Gamma \phi(\mathbf{x}_0)$.  Thus, 
\begin{equation}
r( \mathbf{x}_0)-f( \mathbf{x}_0)||\nabla_\Gamma {\phi}(\mathbf{x}_0)|| \leq 0,
\end{equation} 
and $w$ is a viscosity subsolution. The same argument can be applied to show that $w$ is a viscosity supersolution. Therefore, $w$ is the viscosity solution to \eqref{EikHJBman} and hence $w=u$.
\end{proof}

To summarize, we have shown that we may appropriately extend the control space in the case of isotropic speed and running costs. This has the implication that even though the admissible paths on $T_\epsilon$ have the ``choice" to leave the tangent space of $\Gamma$, the optimal paths in the extended control problem remain on or parallel to $\Gamma.$

\section{Numerical implementation and simulations}\label{sec:sim}

\subsection{Numerical setup}\label{sec:setup}
In the implementation of our new framework, we use a uniform Cartesian grids and denote the grid step size by $h$. $T_\epsilon^{h}$ denotes the part of the grid in the narrow band of radius $\epsilon$ around the given surface. 

We use the Lax-Friedrichs fast sweeping methods in \cite{Kaoetal04} and a high order version in \cite{Zhangetal06} for our simulations. It is also possible to an upwind scheme generalized from \cite{Tsaietal03}. Unless we mention otherwise, we use the standard first order finite differencing in the approximation of the partial derivatives for the Lax-Friedrichs numerical Hamiltonians.\footnote{ All code used to produce the numerical simulations can be found at https://github.com/lindsmart/MartinTsaiExtHJB.}

We present several examples of our new formulation performed on surfaces of co-dimension one in three dimensions. For the first three examples, we compute the solution to the Eikonal equation on the surface. As shown in Section \ref{sec:Eikonalcase}, the equivalent equation on $T_\epsilon$ is \begin{align}\label{eq:exEik} ||B(\mathbf{z},\mu) \nabla v(\mathbf{z})||&=1, \;\; \mathbf{z} \in T_\epsilon\backslash\overline{\mathcal{T}} \\ \label{eq:exEikbc}
v(\mathbf{z})&=\overline{g}(\mathbf{z}), \;\; \mathbf{z} \in \overline{\mathcal{T}}.
\end{align}
Again, $$B(\mathbf{z},\mu)=\sigma_1^{-1}\mathbf{t}_1 \otimes \mathbf{t}_1 +\sigma_2^{-1} \mathbf{t}_2 \otimes \mathbf{t}_2+\mu \mathbf{n}\otimes \mathbf{n}.$$ Since the desired solution, $v$, has been proven to be constant along the normals of the surface, $\Gamma$, $\mathbf{n}\otimes \mathbf{n} \nabla v(\mathbf{z})=0$. Therefore, $\mu$ can be any real number. We let $\mu=1$ in all of our computations.

Approximating the solution to \eqref{eq:exEik}-\eqref{eq:exEikbc} requires the computation of the singular values and vectors of the derivative of the closest point mapping, $ P_\Gamma'$. We defer the discussion of these approximations to Section \ref{sec:PointCloud}. In the last example, we apply our framework to solve an HJB equation with an anisotropic speed function.

\subsubsection{Boundary closure}\label{sec:bdryclosure}

We note that the analytical formulation of the HJB equations does not require boundary conditions on $\partial T_\epsilon$. However, since we are using Cartesian grids we must take careful consideration of the discretization near the boundary, $\partial T_\epsilon^h$. When approximating the partial derivatives of $v$, a neighboring grid node may lie outside of $T_\epsilon^h$. We will call these ghost nodes. In our implementation, we provide a boundary closure procedure to enforce the fact that the solution is a constant along normal function. For each ghost node, we perform the following procedure:
\begin{enumerate}
    \item Project the node into the narrow band.
    \item The value at the ghost node is then calculated by interpolating grid values surrounding the projected point. Formally, we must use an interpolation scheme of order higher to the discretization of the PDE on $T_\epsilon^h$. 
\end{enumerate}

Next, we describe how the boundary closure procedure affects the numerical accuracy. Let $\mathbf{z}_{i}$ be a ghost node, and $\mathbf{z}^\alpha_{i}=\mathbf{z}_{i}-\alpha \mathbf{n}_{i}$ be the projection of the ghost node into $T_\epsilon$. Denote the solution at $\mathbf{z}_i^\alpha$ by $v_{i}^\alpha:=v(\mathbf{z}_{i}^\alpha)$. Suppose that we use a first order scheme to discretize the HJB equation on $T_\epsilon^h$. Then the interpolation used in the approximation of $v_i^\alpha$ should yield at least second order in $h$ accuracy, in order to maintain formally the first order accuracy. This is due to the amplification by a factor of $h^{-1}$ in the finite difference scheme errors of the approximation of the values at the neighboring grid nodes of $\mathbf{z}_i$. 

\subsubsection{Depth of projected points and bounding the thickness of the narrow band}\label{sec:thicknessNB}
When choosing $\alpha$ for the the projected ghost node discussed above, a necessary condition is that the nearby surrounding nodes used to interpolate the value of $v_i^\alpha$ must be in $T_{\epsilon}^h$. In our numerical simulations, we use cubic interpolation to approximate $v({\bf z}^{\alpha})$. Thus, it requires $4^{3}$ nearby grid nodes for three dimensional cases.  It can be shown that if $|d_\Gamma(\mathbf{z}_i)-\alpha|\leq\epsilon-2\sqrt{n}h$, where $n$ is the dimension. Then the inner nodes used to interpolate the value at each ghost node are in $T_{\epsilon}^h$. We suggest to use a projected node as close to the boundary node as possible. Therefore in our numerical simulations, we choose $\alpha=d_\Gamma(\mathbf{z}_i)-(\epsilon-2\sqrt{n}h)$. We discuss the effect of the choice of $\alpha$ in Section \ref{sec:ex1}.

The maximal value of $\epsilon$ is restricted by the curvatures of $\Gamma$. We must have $\epsilon$ smaller than the reach of the surface, which depends on the reciprocal of the maximal curvatures.
The minimal value of $\epsilon$ is determined by the finite difference stencils in approximating the partial derivatives of the PDE and the boundary closure procedure. Both stencils require that the narrow band is sufficiently ``thick" relative to the mesh size $h$ with $\epsilon>2\sqrt{n}h$.

 The resulting numerical discretization of the HBJ equation on $T_\epsilon^h$ is convergent 
 for any $\epsilon$ in the interval described above; i.e. it is convergent for very thin $\mathcal{O}(h)$ narrow bands as well as for thicker one whose widths are independent of $h$.
 We remark that this convergent regime is very different that requiring $\epsilon\sim\mathcal{O}(h^p), 0<p<1$, in the related work of \cite{KublikTsai16}, dealing with singular integrals.

\subsubsection{Approximation errors}
The error computed by the proposed algorithm can be written as the sum of errors corresponding to different approximations:
 \begin{equation}\label{eq:error_sources}
 \mathrm{Error} = E_{model}(\epsilon)+E_{\Delta}(h)+E_{\Gamma}(h)+E_{L_p}(h,\epsilon)\,\, , 
 \end{equation}
 where 
 \begin{itemize}
     \item $E_{model}$ relates to how the surface PDE is approximated in the tubular neighborhood by an extended PDE and the boundary conditions. Our main contribution is in deriving the extension for which $E_{model}\equiv 0$ for any $\epsilon$ that is smaller than the reach of the surface;
     \item $E_{\Delta}$ corresponds to the numerical error for approximating the extended PDE problem. Our extension allows for higher order (in $h$) methods to be applied;
     \item $E_{\Gamma}$ relates to errors in approximating the surface and its geometric properties. 
     This error corresponds to approximation of the closest point mapping to the surface and its curvatures. Since these quantities are computed on based on finite differencing using the grid, the error depends on $h$; 
    \item $E_{L_p}$ is the numerical error in approximating the $L_p$ norm of the computed errors. Our formulation allows for a very accurate approximation to this error term. 
 \end{itemize}

\subsubsection{Application to solving HJB on point clouds}\label{sec:PointCloud}
\begin{figure}
\centering
\includegraphics{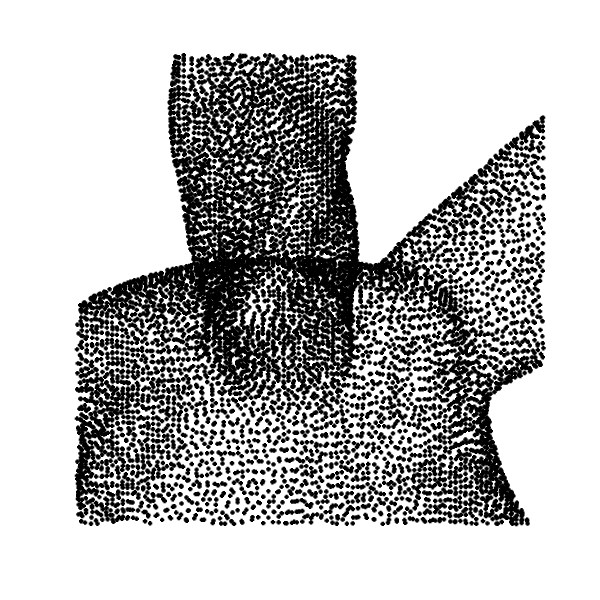}
\caption{Example of a point cloud for the Stanford bunny zoomed in near the top of the head.}
\label{fig:bunnyptcld}
\end{figure}
We now describe how we compute the closest point mapping when the surface is represented as a point cloud. (See Figure \ref{fig:bunnyptcld}.) Denote the point cloud by $\Gamma_K\subset\Gamma$, with $K$ indicating the total number of points in the set. Here, we shall assume that the $K$ points uniformly distribute on the surface. The purpose of the exposition here is to point out a promising application of the method to point cloud data. A full analysis of the algorithm in this area  warrants a separate paper.

We approximate $P_\Gamma$ using the following strategy:
\begin{enumerate}
    \item For each $\mathbf{z}_i \in T_\epsilon^{h}$, estimate $\Gamma_K$ locally on the grid nodes. Call the local surface associated to the point $\mathbf{z}_i,$ $\Gamma_K(\mathbf{z}_i)$.
    \item  Compute the closest point of $\mathbf{z}_i$ on the local surface, $\Gamma_K(\mathbf{z}_i).$
\end{enumerate}

 In our implementation, we use biquadratic interpolation at each $\mathbf{z}_i$ to compute $\Gamma_N(\mathbf{z}_i)$. Another option to locally estimate the surface is least squares as in \cite{LiangZhao2013}. We then use Newton's method to obtain the closest point of each $\mathbf{z}_i$ on  $\Gamma_N(\mathbf{z}_i).$ Denote the closest point mappings by $P_{\Gamma_N}^{\textrm{int}}(\mathbf{z}_i)$. Then on each grid node $\mathbf{z}_i$, the equation is discretized as usual, with $P_\Gamma^\prime(\mathbf{z}_i)$ being approximated by finite differences of $ P^{\textrm{int}}_{\Gamma_N} (\mathbf{z}_i).$ We use a fourth-order finite difference scheme to estimate $P_\Gamma'$ and singular value decomposition to obtain the singular values and singular vectors of $P_\Gamma'$. We show the effect of this approximation procedure in Section \ref{sec:ex1}.

In any case, the term, $E_\Gamma$, in \eqref{eq:error_sources} now depends on $1/K$, which corresponds to the uniform spacing of the points. Therefore, with a fixed point cloud, refinement of Cartesian grids eventually will not further improve the accuracy in the numerical solutions when compared to the solution on the idealized smooth surface sampled by the point cloud. In Section \ref{sec:samplingdensities}, we present some results revealing the effect of different surface sampling density.

If noise is introduced to the point cloud, and if the amount of perturbation in the original "surface normals" is significant, more sophisticated surface fitting is required; e.g. adopting a similar strategy proposed in \cite{LiangZhao2013}.  We will investigate this important issue in a future manuscript. Nevertheless, in Section \ref{sec:samplingdensities}, we present some results revealing the effect of noise in the point cloud.

\subsection{Example 1.1}\label{sec:ex1}
First, we present a numerical convergence study. We approximate the solution to \eqref{eq:exEik}-\eqref{eq:exEikbc}, where $\Gamma$ is a sphere centered at $(0,0,0),$ with radius, $r_0=0.5$, and $\mathcal{T}$ is a point on the sphere. The exact singular values and vectors in $B(\mathbf{z},1
)$ are used. In the case of a sphere \begin{equation}\label{eq:sigs}\sigma_1=\sigma_2=1-\frac{r(\mathbf{z})-r_0}{r(\mathbf{z})}=\frac{r_0}{r(\mathbf{z})},\end{equation} where $r(\mathbf{z})$ is the radius of the sphere going though $\mathbf{z}$ centered at $(0,0,0)$. Figure \ref{fig:sphere} shows the distance function to a point on the sphere for two different view points.

Since the solution is constant along the normals of $\Gamma$, we can easily estimate the $L_1$ error. Using the formulation in \cite{KublikTsai16}, the estimation of the $L_1$ error is given by the following formula:
\begin{align*}
||u-u^h||_{L_1(\Gamma)}\approx \sum_{\mathbf{z}_i \in T_\epsilon^h} (u(P_\Gamma \mathbf{z}_i)-v_i^h)K_\epsilon(d_\Gamma(\mathbf{z}_i))J(\mathbf{z}_i)h^3,
\end{align*}
where $u$ is the exact solution and $u^h$ is the approximated solution on $\Gamma$ using the approximate solution, $v_i^h$, of the extended equation with grid size $h$.
Here, $$K_\epsilon(d_\Gamma(\mathbf{z}_i))=A *\text{exp}\bigg (\frac{2}{d_\Gamma(\mathbf{z}_i)^2-1} \bigg) $$ where $A =7.513931532835806$ 
and 
$$J(\mathbf{z}_i)=\sigma_1\sigma_2,$$ where $\sigma_1,\sigma_2$ are given in \eqref{eq:sigs}.

Because all the characteristics of the solution emanate from the source point on the sphere, the errors at the source point will propagate through the solution in the entire computational domain. Therefore, we initialize the solution near $\overline{\mathcal{T}}$ with the exact solution if $u(\mathbf{z}) <0.2$ The $L_1$ and $L_\infty$ errors are reported in Table \ref{tab:converg}. We see that first order convergence is achieved in the $L_1$ error. However, because of the singularity of the solution at the point opposite the source point on the sphere, the order in the $L_\infty$ error is less than one. The results verify that the boundary closure procedure does not influence the overall order of the scheme.

 We report the number of iterations it takes for the $L_\infty$ norm of the difference in the solution between successive iterations is less than $1.0e-13$. One iteration includes one sweep of the grid in each of the eight sweeping directions. The timing for the numerical experiments scale linearly to the number of grid nodes in the narrow band $T_\epsilon$ and the number of iterations used in the Lax-Friedrichs scheme. The computations were performed on a MacBook laptop with a 1.6 GHz Intel CPU on a single core using Julia.  One Lax-Friedrichs iteration took on average 0.34260 seconds for the case $N=101$ in Table 1. We note that our code was not optimized for efficiency.

\renewcommand{\arraystretch}{1.5}
\begin{table}
\caption{Errors and orders of accuracy for the distance function on a sphere, computed by the first order Lax-Friedrichs sweeping algorithm. We use $h=2/(N-1)$ and $\epsilon = 4h$.}

\begin{center}

\begin{tabular}{|r|r|r|r|r|r|r|}
\hline
$N$ & number of points in $T_\epsilon$ & $L_1$ error & order     & $L_\infty$ error & order    & iterations \\ \hline
101 & 63302                        & 0.02319     &           & 0.03741          &          & 32         \\ \hline
201 & 252310                       & 0.01138     & 1.0345 &  0.02213          & 0.76291 & 49         \\ \hline
301 & 566458                       & 7.5443e-3   & 1.0158    & 0.01611          & 0.78624  & 65         \\ \hline
401 & 1571974                      & 5.6562e-3   & 1.0041    & 0.01280          & 0.79978  & 80         \\ \hline
\end{tabular}


\end{center}
\label{tab:converg}
\end{table}

Another advantage of our setup is that we can use existing high order methods to compute the solution to HJB equations on surfaces. In Table \ref{tab:convergehighorder}, we present the $L_1$ and $L_\infty$ errors and orders for the third order method given in \cite{Zhangetal06}. We initialize the solution near $\overline{\mathcal{T}}$ to the exact solution if $u(\mathbf{z}) < 0.2.$ We see that third order is achieved in the $L_1$ error. Again, because of the singularity in the solution at the point opposite the source point on the sphere, the order in the $L_\infty$ error is first order.

\renewcommand{\arraystretch}{1.5}
\begin{table}
\caption{Errors and orders of accuracy for the distance function on a sphere, computed by the third order Lax-Friedrichs sweeping algorithm. We use $h=2/(N-1)$ and $\epsilon = 11h$.}

\begin{center}

\begin{tabular}{| r|r|r|r|r|r|r|}
 \hline

 $N$ & number of points in $T_\epsilon$ & $L_1$ error & order & $L_\infty$ error &  order & iterations\\ 
 \hline
101 & 183810 & 5.1129e-4 &  & 0.02779 &  & 65            \\ \hline

201 & 702626 & 6.6194e-5 & 2.9706 & 0.01137 & 1.2987 & 92            \\ \hline

301 & 1566014 & 2.0373-5 & 2.9182& 7.1199e-3 & 1.1587 & 100
\\ \hline

401 & 2776370 & 8.5992e-6 & 3.0069 & 5.1806e-3 & 1.1085 & 119
\\\hline


\end{tabular}

\end{center}
\label{tab:convergehighorder}
\end{table}

\renewcommand{\arraystretch}{1.5}

\subsection{Example 1.2}\label{sec:samplingdensities}
In this example, we compute the distance function on the sphere when the sphere is represented as a point cloud. We show the effect of our method when using different uniform samplings of the sphere.\footnote{The uniform samplings of the sphere were computed using the code provided at \url{https://github.com/AntonSemechko/S2-Sampling-Toolbox}. } The results are tabulated in Table~\ref{tab:covergepointcloud} when the sampling of the sphere has $400,800,$ and $1600$ points. We again initialize the solution by exact solutions around the same neighborhood of $\overline{\mathcal{T}}$ as in Section \ref{sec:ex1}.   By comparing the $L_1$ error from Table \ref{tab:converg} for $N=101$ and the $L_1$ error in Table \ref{tab:covergepointcloud} for $N=101$ and no noise, it appears that the dominating source of error is the numerical discretization of the PDE. As the grid is refined, we see that it appears the error related to approximating the surface begins to dominate. This is apparent by comparing the $L_1$ error from Table \ref{tab:converg} for $N=301$ and the $L_1$ error from Table \ref{tab:covergepointcloud} for $N=301$ with no noise. Notice as the point cloud becomes more dense and the accuracy of the surface approximation increases, it appears the dominating error is again due to the numerical discretization of the PDE.

In Table \ref{tab:covergepointcloud}, we also consider when noise is introduced to the point clouds of different sampling densities. The noise is applied by adding $\delta*(-1+2*\text{rand}())/2$ to each point in the point clouds, where $\delta = 0.001, 0.005$ and $\text{rand()}$ is a uniformly distributed random number in $(0,1)$. We can see if the noise is too large the error begins to deteriorate and we need a more sophisticated surface fitting algorithm in our method.

\begin{table}
\caption{$L_1$ errors for the first order Lax-Friedrichs scheme using different densities of uniformly sampled point cloud representations of the sphere and different noises applied to the point clouds. The noise is applied by adding $\delta*(-1+2*\text{rand}())/2$ to each point in the point clouds.  We use $h=2/(N-1)$, $\epsilon=4h$.}

\begin{center}

\begin{tabular}{|c|r|r|r|r|}
\hline
\multicolumn{2}{|l|}{}                                 & \multicolumn{3}{c|}{ $L_1$ errors using $\Gamma_K$} \\ \hline
\multicolumn{1}{|c|}{$\delta$} & \multicolumn{1}{c|}{N}   & $K=400$          & $K=800$          & $K=1600$         \\ \hline
\multirow{3}{*}{0}          & 101                      & 0.02332      & 0.02317      & 0.02318      \\ \cline{2-5} 
                            & 201                      & 0.01172      & 0.01144      & 0.01139      \\ \cline{2-5} 
                            & 301                      & 8.0413e-3    & 7.6406e-3    & 7.5649e-3    \\ \hline
\multirow{3}{*}{0.001}      & 101                      & 0.02349      & 0.02321      & 0.02313      \\ \cline{2-5} 
                            & \multicolumn{1}{l|}{201} & 0.01192      & 0.01153      & 0.01148      \\ \cline{2-5} 
                            & \multicolumn{1}{l|}{301} & 8.2866e-3    & 7.7936e-3    & 7.75321e-3   \\ \hline
\multirow{3}{*}{0.005}      & 101                      & 0.02464      & 0.02364      & 0.02404      \\ \cline{2-5} 
                            & \multicolumn{1}{c|}{201} & 0.01389      & 0.01339      & 0.01531      \\ \cline{2-5} 
                            & \multicolumn{1}{c|}{301} & 0.01082      & 0.01265      & 0.01279      \\ \hline
\end{tabular}

\end{center}
\label{tab:covergepointcloud}
\end{table}

\subsection{Example 1.3}
Next, we study the affect that the depth of projected ghost nodes has on the overall error for the distance function on the sphere. The set up is the same as in the previous section. Here, we choose a $101^3$ sized grid with $\epsilon=11h$ where $h=2/100$. We estimate the $L_1$ error for when the boundary closure procedure is carried out at four different depths. Recall that a ghost node, $\mathbf{z}_i$, is outside of $T_\epsilon^h$ and is projected into $T_\epsilon$ along the normal of the surface at that point, i.e., 
$$\mathbf{z}_i \mapsto \mathbf{z}_i-\alpha \mathbf{n}_i.$$ Table \ref{tab:alphadepth} shows the $L_1$ error for the depths: $$\alpha=d(\mathbf{z})-(10-2\sqrt{3})h, d(\mathbf{z})-3h, d(\mathbf{z}), d(\mathbf{z})+3h.$$ We can see that the error the solution is not very sensitive to the choice of depth. In our simulations, we chose $\alpha=\epsilon -2\sqrt{3}h$.

\begin{table}
\begin{center}

\caption{Comparison of different choices of $\alpha$ for the sphere on a $101^3$ grid. Here, $\epsilon=11h$ and $h=2/100.$}
\begin{tabular}{| r|r|r|r|r |} 
 \hline

$\alpha$ & $d(\mathbf{z})-(11-2\sqrt{3})h$& $d(\mathbf{z})-3h$ & $d(\mathbf{z})$ &$d(\mathbf{z})+3h$\\
\hline
$L_1$ error & 0.02298 & 0.02303 & 0.02320 & 0.02349 \\\hline

\end{tabular}
\label{tab:alphadepth}
\end{center}

\end{table}
\subsection{Example 1.4}
Finally, we compare the following: (1) our method on the sphere, using exact singular values and vectors as in the set up for Table \ref{tab:converg}; (2) our method with the point cloud procedure described in \ref{sec:PointCloud} and finite differences to estimate $P_\Gamma'$; (3) the method in \cite{MemoliSapiro01}. Since method (3) requires $\epsilon=Ch^\gamma$ where $C=\sqrt{3}$ and  $\gamma \in (0,1)$, to fairly compare the error of the methods, we increase the width of the narrow band to $h=2h^{0.7},$  which is used in the convergence analysis of (3) in \cite{MemoliSapiro01}. We initialize the experiments using methods (1), (2), and (3) with exact values given in a box of length $h$ around the source point/points in the initial sets. We also are not able to compute the $L_1$ error for method (3) since the solution is not constant along the normals of $\Gamma$. Therefore, we can only report the $L_\infty$ error, and we use trilinear interpolation to approximate the solutions on $\Gamma$ for (3). The errors are reported in Table \ref{tab:comparison}. We can also see that estimating $P_\Gamma$ from a point cloud 
and $P_\Gamma'$ from finite differences does not greatly influence the $L_1$ error of the method. 

\renewcommand{\arraystretch}{1.5}
\begin{table}
\caption{Comparison of methods on a $101^3$ grid: (1) Our method using exact closest point mapping and exact singular values and vectors; (2) Our method using closest point mapping sampled from a an equally distributed point cloud of 400 points and finite differences to compute singular values and vectors; (3) method in \cite{MemoliSapiro01}. Here $\epsilon=2h^{0.7}$ and $h=2/100$.}
\begin{center}
\begin{tabular}{ |r|r|r|r |} 
\hline

Method  & (1) & (2) & (3) \\\hline

 \hline
$L_\infty$ error & 0.03226 & 0.03334 & 0.17072 \\\hline

$L_1$ error      & 0.07787 & 0.07667 & \\\hline

\end{tabular}
\label{tab:comparison}
\end{center}

\end{table}

\subsection{Example 2}

\begin{figure}
\centering
\includegraphics{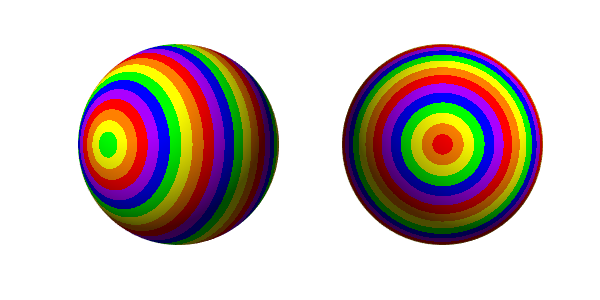}

\caption{Two view points of the distance function on a sphere.}
\label{fig:sphere}
\end{figure}

\begin{figure}
\centering
\includegraphics{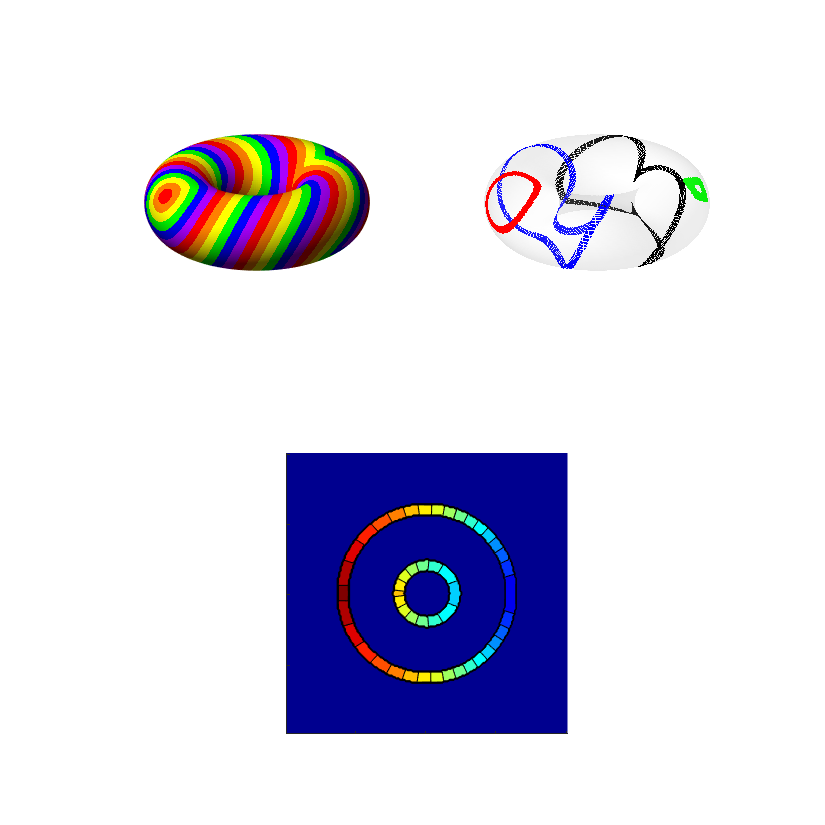}

\caption{Top left: Distance function on a torus. Top right: Points from the corresponding point cloud whose distance from the source point lies in the interval (0.09,1.1), (0.29,3.1), (0.49,5.1), or (0.69,7.1). Bottom: Solution slice at z=.5 with contours showing that the solution is indeed constant along normal.}
\label{fig:torus}
\end{figure}

\begin{figure}
\centering
\includegraphics{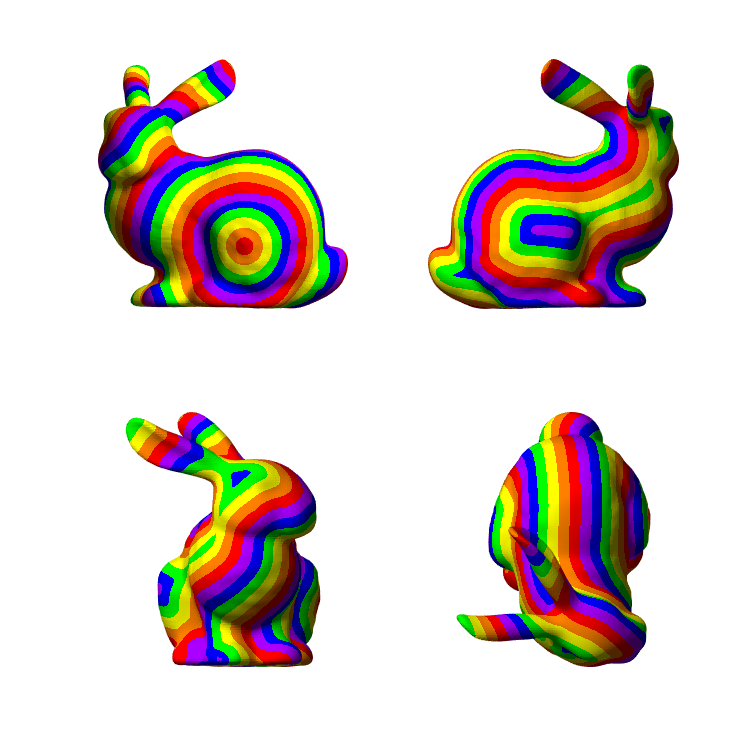}
\caption{Distance function on the Stanford bunny with four viewpoints.}
\label{fig:bunny}
\end{figure}

\begin{figure}
\centering
\includegraphics{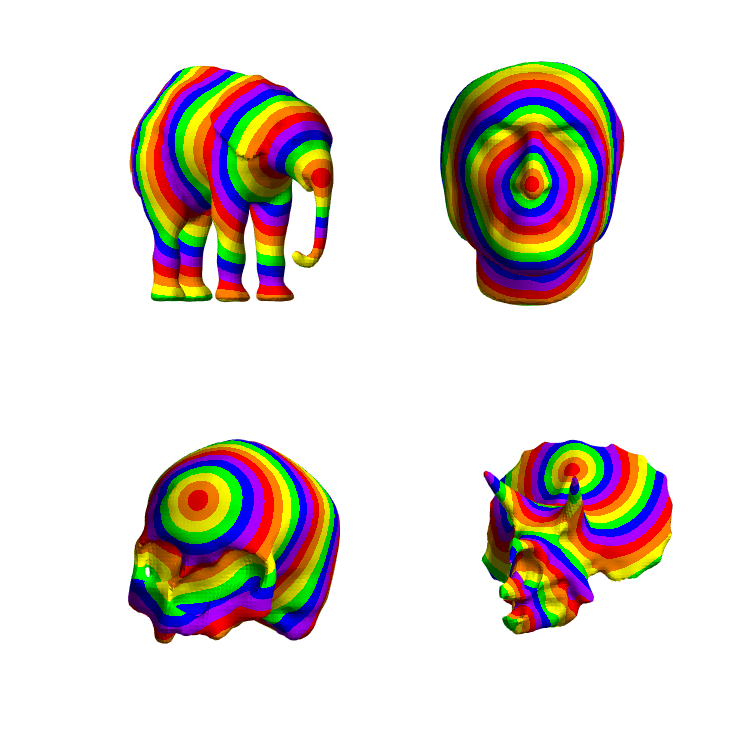}
\caption{Distance function on a elephant, mask, human skull, and dinosaur skull.}
\label{fig:otherfigs}
\end{figure}

In Figures \ref{fig:torus}, \ref{fig:bunny}, and \ref{fig:otherfigs}, the distance function to a source point is shown on various surfaces. The contours shown are equally spaced and parallel. The solutions for the torus and bunny are computed on a $201^3$ grid, and a $101^3$ grid is used for the elephant, mask, human skull, and dinosaur skull. All computations use a narrow band width of $\epsilon=4h$ and point cloud representations for the surfaces. The number of points in the point clouds of the torus\footnote{The point cloud for the torus was generated using the standard parametrization of a torus.}  and bunny\footnote{The point cloud for the Stanford bunny is generated from a refinement of the triangulated Stanford bunny from https://casual-effects.com/data/ \cite{McGuire2017}.}  are 178,350 and 228,096, respectively. For the elephant, mask, human skull, and dinosaur skull\footnote{The point clouds for the surfaces in Figure \ref{fig:otherfigs} were generated from the triangulations downloaded at \url{https://www.myminifactory.com/scantheworld/}.}, we used point clouds with 65,292, 1,199,988, 234,618, and 139,491 points, respectively.  In Figure \ref{fig:torus}, a cross section of the solution is displayed to show that the solution is indeed constant along the normals of the surface. Our new framework also allows us to sort point clouds. Figure \ref{fig:torus} displays level ``belts'' of the point cloud, i.e., points in the point cloud whose distance from the source point lies in given interval. 

Another advantage of our formulation is that when we compute the characteristics, known as geodesics, of the Eikonal equation via the extended Eikonal equation on the narrowband, the paths remain on the surface if the initial point lies on $\Gamma$.  Since the solution to \eqref{eq:exEik}-\eqref{eq:exEikbc} is constant along the normals of $\Gamma$, the gradient always belongs to the tangent spaces of $\Gamma$ or its parallel surfaces, $\Gamma_\eta$. We have that $B(\mathbf{x},\mu) \mathbf{a}=\mathbf{a}$ for $\mathbf{x} \in \Gamma$ and $\mathbf{a} \in A_\mathbf{x}=T_{\mathbf{x}}\Gamma \cap S^{n-1}$ since $\sigma_1=\sigma_2=1$. Thus, if the initial point is $\mathbf{z} \in \Gamma$, the geodesic on the surface can be obtained by solving the dynamical system \begin{align}
\frac{d{\mathbf{y}}}{dt}(t)&=\frac{-\nabla v(\mathbf{y}(t))}{||\nabla v(\mathbf{y}(t))||}, \; t>0 \\ 
{\mathbf{y}}(0)&=\mathbf{z}, \mathbf{z}\in \Gamma .
\end{align}

\begin{figure}
\centering
\includegraphics{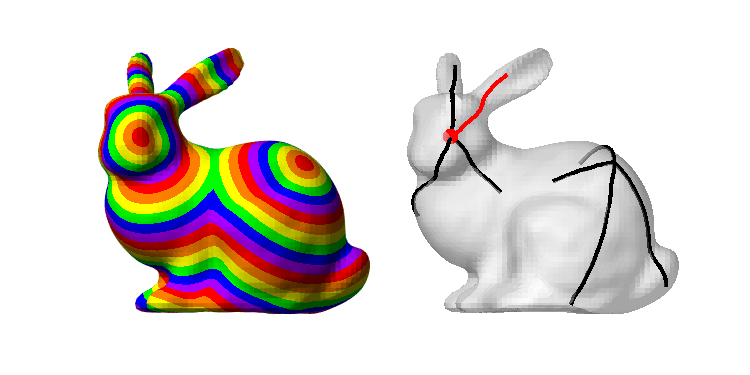}
\caption{Left: The distance function for two source points. Right: Eight geodesics computed on the Stanford bunny.}
\label{fig:bunnygeo}
\end{figure}
In Figure \ref{fig:bunnygeo}, we consider the case when we have two source points and display the solution and some geodesics. Recall that in Theorem \ref{radial0} we showed that paths with equivalent initial points remain equivalent for all time. We show a visualization of this property in Figure \ref{fig:parallelpaths}. In the next section we will compare geodesics of the HJB equation given different anisotropic speed functions. 
\begin{figure}
\centering
\includegraphics{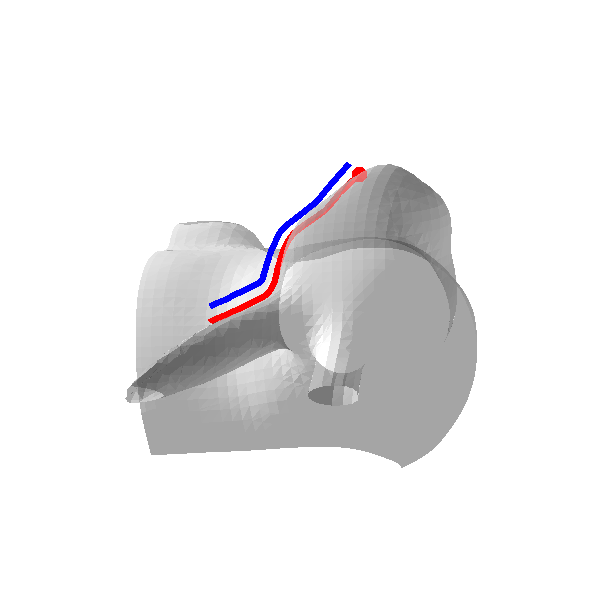}
\caption{Verification of Theorem \ref{radial0}, i.e., paths with equivalent starting points stay equivalent for all time. We plot the red geodesic from Figure \ref{fig:bunnygeo}, and the initial point of the blue parallel path is offset by 0.015 along the normal of the bunny at that point.}
\label{fig:parallelpaths}
\end{figure}
\subsection{Example 3}

\begin{figure}
\centering
\includegraphics{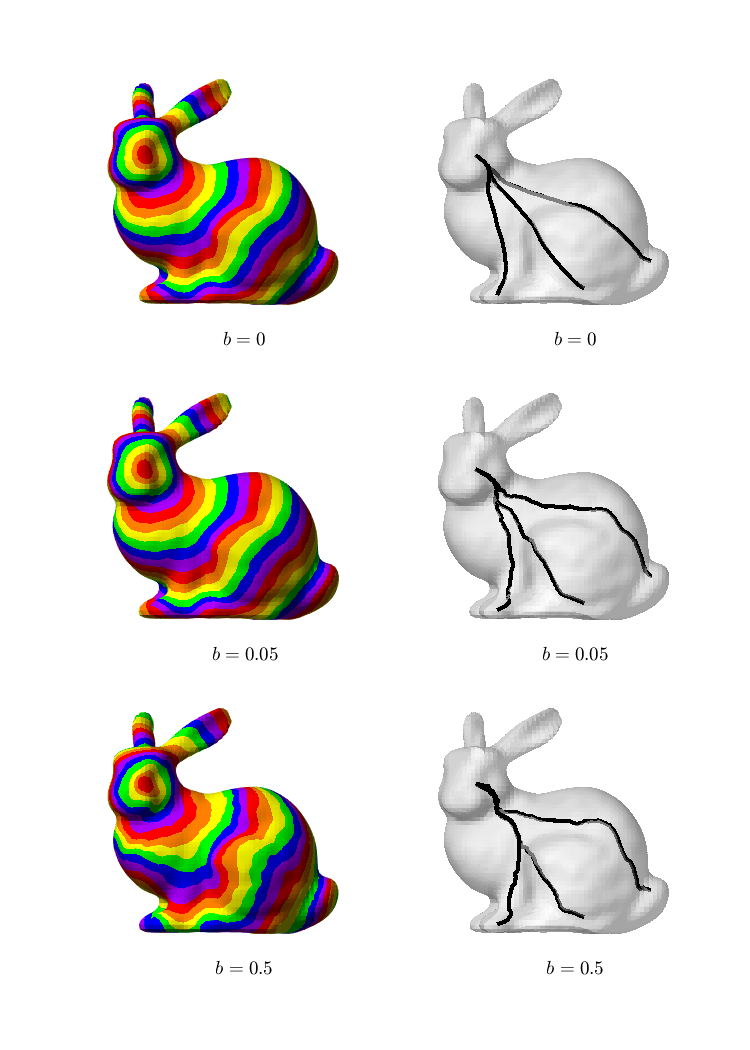}
\caption{Anisotropic example for the curvature based speed function showing the contours and geodesics for values of $b=0,0.05,0.5$.}
\label{fig:bunnyani}
\end{figure}

Finally, we test implement the framework on an anisotropic speed function. We solve 

\begin{align} \label{HJBmanex}
\min_{\mathbf{a}\in A_\mathbf{z}} \bigg\{\nabla v(\mathbf{z}) \cdot \overline{f}( \mathbf{z},\mathbf{a})B(\mathbf{z},\mu)\mathbf{a} +1 \bigg \}&=0 , \; \; \mathbf{z} \in T_\epsilon\backslash \overline{\mathcal{T}},\\
 \label{HJBmanexbc} 
 v(\mathbf{z})&=\overline{g}(\mathbf{z}), \; \; \mathbf{z}\in \overline{\mathcal{T}},
\end{align}
where $f: \Gamma\times A \to \mathbb{R}$ is a curvature based speed function on the surface. We use the speed function proposed in \cite{Seonetal12}. The speeds are fast on low curvature areas of the surface and slow on high curvature areas. This means that the corresponding ``shortest" paths traverse areas of the lowest curvatures. These paths are typically longer than the geodesics of the surface.  

The normal curvature of $\Gamma$ at $\mathbf{x}$ in the direction $\mathbf{a}$ is given by $$\kappa_\mathbf{a}(\mathbf{x})=\mathbf{a}^T\left[ \begin{array}{cc}
   \kappa_1 & 0  \\
   0 & \kappa_2 \\
 
  \end{array}  \right]\mathbf{a}$$
  where $\kappa_1$ and $\kappa_2$ are the principal curvatures of $\Gamma$ at $\mathbf{x}$.
Then the curvature-minimizing  speed function is given by
   $$f(\mathbf{x},\mathbf{a})=\mbox{exp}(-b ||\kappa_\mathbf{a}(\mathbf{x})||),$$
where $b$ is a positive constant. When $b=0$, \eqref{HJBmanex} reduces Eikonal equation on the surface. A larger value of $b$ corresponds to a greater difference in speeds between areas of low and high curvature. We display the solutions for varying $b$ values to a source point on the bunny in Figure \ref{fig:bunnyani}. 

In this example, the set $V(\mathbf{x})=\{f(\mathbf{x},\mathbf{a}) \mathbf{a} \; | \; \mathbf{a} \in A_\mathbf{z}\}$  is not neccesarily convex for all $\mathbf{x}$ in $\Gamma.$ Therefore, an optimal control may not exist. However, we can still extract suboptimal paths, called anisotropic geodesics, whose total cost is arbitrarily close to the value function at the starting point. Just as in the isotropic case, the anisotropic geodesics computed from the extended HJB equation on the narrow band will lie on the surface since $B(\mathbf{x}, \mathbf{a})\mathbf{a}=\mathbf{a}$ for $\mathbf{x} \in \Gamma$. The paths can easily be extracted because we compute the minimizing control at each grid node when solving \eqref{HJBmanex}-\eqref{HJBmanexbc}. Once we have the optimal (or suboptimal) control values, $\mathbf{a}^*(\cdot)$, we then solve dynamical system \begin{align*}
\frac{d{\mathbf{y}}}{dt}(t)&={f}({\mathbf{y}}(t),\mathbf{a}^*(t))\mathbf{a}^*(t), \; t>0 \\ 
{\mathbf{y}}(0)&=\mathbf{z},\; \mathbf{z} \in \Gamma.
\end{align*} We plot three anisotropic geodesics for each $b$ value in Figure \ref{fig:bunnyani}. We can see as $b$ increases the Euclidean distances of the paths are longer, and the paths start to seek out the narrow valleys of the bunny.

\section{Summary and conclusion}
In this paper, we presented a new formulation to compute solutions of a class of HJB equations on smooth hypersurfaces. We extend the HJB equation's associated optimal control problem from the surface to an equivalent problem defined in a sufficiently "thin" narrow band around the surface, in the embedding Euclidean space. The extension was done so that the resulting value function is the constant normal extension of the value function defined in the optimal control problem on the surface. We presented the formulations for the general anisotropic equation and showed that the viscosity solution of the HJB equation on the narrow band is the constant normal extension of the viscosity solution on the surface, independently of the optimal control problems. We also presented the isotropic case and showed there is no need to restrict the control space in order to have an equivalent formulation.  The proposed approach is independent of surface representation and can be used to compute and define optimal control problems on uniformly distributed point clouds sampled from some smooth surface. It is  also clear that our proposed extension approach can be applied in to time dependent equations arising from the finite horizon control problems. Together with \cite{ChuTsai18}, our extension approach provides a good framework for solving to mean field games to high order accuracy on complicated and non-parametrized surfaces.

With this new framework, we are able to use Cartesian grids and the existing methods for computing solutions to HJB equations in Euclidean space on a very thin narrow band to solve HJB equations on surfaces coupled with  a simple boundary closure procedure. Our numerical examples verify that the boundary closure procedure does not influence the overall order of the method. We also show that our formulation allows one to easily solve the surface HJB equations to high order accuracy.

\section*{Acknowledgment}
The authors are supported partially by National Science Foundation Grant DMS-1720171. Tsai also thanks National Center for Theoretical Study, Taipei for hosting his visits, in which some of the ideas presented in this paper originated.

\bibliographystyle{abbrv}

\bibliography{HJBManifolds}

\end{document}